\documentclass[11pt]{amsart}

\usepackage{amsthm,amsmath,amssymb,tikz}
\usepackage{graphicx}
\usetikzlibrary{positioning}

\usepackage[margin=3.25cm]{geometry}

\RequirePackage{hyperref}
\RequirePackage{hypernat}

\newtheorem{theorem}{Theorem}[section]
\newtheorem{lemma}[theorem]{Lemma}
\newtheorem{proposition}[theorem]{Proposition}
\newtheorem{corollary}[theorem]{Corollary}
\theoremstyle{definition}

\newtheorem{problem}[theorem]{Problem}
\newtheorem{conjecture}[theorem]{Conjecture}
\newtheorem{definition}[theorem]{Definition}
\theoremstyle{remark}
\newtheorem{example}[theorem]{Example}
\newtheorem{remark}[theorem]{Remark}

\newcommand{\Top}{\text{top}}
\newcommand{\bi}{\leftrightarrow}

\DeclareMathOperator{\pa}{pa}
\DeclareMathOperator{\sib}{sib}
\DeclareMathOperator{\an}{an}

\numberwithin{equation}{section}
\numberwithin{figure}{section}

\theoremstyle{definition}

\makeatletter
\newtheorem*{rep@theorem}{\rep@title}
\newcommand{\newreptheorem}[2]{%
\newenvironment{rep#1}[1]{%
 \def\rep@title{#2 \ref{##1}}%
 \begin{rep@theorem}}%
 {\end{rep@theorem}}}
\makeatother

\newreptheorem{theorem}{Theorem}
\newreptheorem{lemma}{Lemma}

\title[Nested Covariance Determinants]{Nested Covariance Determinants
  and Restricted Trek Separation in Gaussian Graphical Models}

\author{Mathias Drton}
\address[M.~Drton]{Department of Statistics\\
University of Washington}
\email[M.~Drton]{md5@uw.edu}
\author{Elina Robeva}
\address[E.~Robeva]{Department of Mathematics\\
MIT}
\email[E.~Robeva]{erobeva@mit.edu}
\author{Luca Weihs}
\address[L.~Weihs]{Department of Statistics\\
  University of Washington}
\email[L.~Weihs]{lucaw@uw.edu}

\keywords{Conditional independence, covariance matrix, graphical
  model, trek separation, Verma constraint}

\begin{document}

\begin{abstract}
  Directed graphical models specify noisy functional relationships
  among a collection of random variables.  In the Gaussian case, each
  such model corresponds to a semi-algebraic set of positive definite
  covariance matrices.  The set is given via a
  parametrization, and much work has gone into obtaining an implicit
  description in terms of polynomial (in-)equalities.  Implicit descriptions
  shed light on problems such as parameter identification, model
  equivalence, and constraint-based statistical inference.  For models
  given by directed acyclic graphs, which represent settings where all
  relevant variables are observed, there is a complete theory: All
  conditional independence relations can be found via graphical
  $d$-separation and are sufficient for an implicit description.  The
  situation is far more complicated, however, when some of the
  variables are hidden (or in other words, unobserved or latent).
  We consider models associated to mixed graphs that capture the
  effects of hidden variables through correlated error terms.  The
  notion of trek separation explains when the covariance matrix in
  such a model has submatrices of low rank and generalizes
  $d$-separation. However, in many cases, such as the infamous Verma
  graph, the polynomials defining the graphical model are not
  determinantal, and hence cannot be explained by $d$-separation or
  trek-separation.  In this paper, we show that these constraints often
  correspond to the vanishing of nested determinants and can be
  graphically explained by a notion of {\em restricted trek
    separation}.
  % Gaussian graphical models are semi-algebraic sets formed by
  % intersecting the cone of positive definite matrices with the
  % vanishing set of several polynomial equations. Knowing this set
  % allows for performing constraint-based inference in these models.
  % For a directed acyclic graph, these polynomial equations are
  % determinants of submatrices of the covariance matrix corresponding
  % to conditional independence statements of the variables, and to
  % d-separation in the graph.  More generally, for a bidirected graph,
  % submatrices with low rank correspond to generalizations of
  % conditional independence constraints on collections of random
  % variables, and graph-theoretically these are represented by
  % trek-separation.  However, in many cases, such as the infamous Verma
  % graph, the polynomials defining the graphical model are not
  % determinantal, and hence cannot be explained by d-separation or
  % trek-separation.  In this paper, we show that these constraints are
  % in fact given by the vanishing of nested determinants, which are
  % graphically explained by a simple notion of {\em generalized trek
  %   separation}.
\end{abstract}
\maketitle

\section{Introduction}
\label{sec:introduction}

Let $G=(V,\mathcal{E})$ be a directed graph with finite vertex set $V$ and edge
set $\mathcal{E}\subseteq V\times V$.  The edge set is always assumed to be free
of self-loops, so $(i,i)\notin \mathcal{E}$ for all $i\in V$.  For each vertex
$i$, define a set of parents $\pa(i)=\{j\in V: (j,i)\in \mathcal{E}\}$.  The
graph $G$ induces a statistical model for the joint distribution of a
collection of random variables $X_i$, $i\in V$, indexed by the graph's
vertices.
% The model hypothesizes that each variable $X_v$ is a function of the
% parent variables $X_{\pa(v)}=(X_w:w\in\pa(v))$ and an independent
% noise term.
The model hypothesizes that each variable is a function of the parent
variables and an independent noise term.  In this paper we consider
the Gaussian case, in which the functional relationships are linear so
that
\begin{align}
  \label{eq:structural-general}
  X_i &= \lambda_{0i} + \sum_{j\in\pa(i)} \lambda_{ji} X_j + \epsilon_i, \quad i\in V,
        % f_i(X_{\pa(i)}, \epsilon_i), \quad i\in V,
\end{align}
where the $\epsilon_i$, $i\in V$, are independent and centered
Gaussian random variables.  The coefficients $\lambda_{0i}$ and
$\lambda_{ji}$ are unknown real parameters that are assumed to be such
that the system~(\ref{eq:structural-general}) admits a unique solution
$X=(X_i:i\in V)$.  Typically termed a system of structural equations,
(\ref{eq:structural-general}) specifies cause-effect relations whose
straightforward interpretability is behind the wide-spread use of the
models \cite{sgs,pearl:book}.

The random vector $X$ that solves~(\ref{eq:structural-general})
follows a Gaussian distribution whose mean vector may be arbitrary
through the choice of the parameters $\lambda_{0i}$ but whose
covariance matrix is highly structured.  The model obtained
from~(\ref{eq:structural-general}) thus naturally corresponds to the
set of covariance matrices, which we denote by $\mathcal{M}(G)$.  This
set is given parametrically with each covariance being a rational or
even polynomial function of the parameters $\lambda_{ji}$ and the
variances of the errors $\epsilon_i$, as we detail in
Section~\ref{sec:background}.

While a parametrization is useful to specify a distribution and to
optimize the likelihood function, many statistical
problems can only be solved with some understanding of an
implicit description.  In our setting, an implicit description of the
model amounts to a semi-algebraic description of the set of covariance
matrices that belong to the model through polynomial equations and inequalities, and a combinatorial
criterion on the graph which specifies how to obtain them.  Specific
problems that can be addressed through such an implicit description
include model equivalence, parameter identification, and
constraint-based statistical inference.  We refer the reader to the
recent work of \cite{VanOmmenMooij_UAI_17} and the reviews of \cite{D}
and \cite{drt17}.

If the underlying graph $G$ is an acyclic digraph, also termed a
directed acyclic graph (DAG), then probabilistic conditional
independence yields an implicit description of $\mathcal{M}(G)$
\cite{lauritzen:1996,studeny:book}.  For a Gaussian joint
distribution, conditional independence corresponds to the vanishing of
special subdeterminants of the covariance matrix, namely,
subdeterminants that are almost principal in the sense that the row
and the column index sets agree in all but one element
\cite[Chap.~3.1]{MR2362722,loas}.  The conditional
independences holding in all distributions in the given model can be
found graphically using the concept of $d$-separation.  It follows in
particular that two DAGs $G$ and $H$ give rise to the same model
$\mathcal M(G) = \mathcal M(H)$ if and only if $G$ and $H$ have the
same $d$-separation relations.  This combinatorial criterion can be
simplified to yield an efficient algorithm:
$\mathcal M(G) = \mathcal M(H)$ if and only if $G$ and $H$ have the
same skeleta and the same sets of unshielded colliders
\cite{Frydenberg90,VermaPearl90}.

While DAG models are well-understood, they only pertain to problems
where all relevant variables are observed.  A long-standing program in
the fields of graphical modeling and causal inference seeks to develop
combinatorial solutions to problems such as model equivalence in
settings with hidden/latent variables.  Mathematically, if only the
variables indexed by a set $A\subset V$ are observed while those
indexed by $V\setminus A$ are hidden, then the covariance matrices in the set
$\mathcal{M}(G)$ are to be projected on their principal $A\times A$
submatrix.  It is well known that conditional independence is no
longer sufficient for implicit model description after such a
projection.

\begin{figure}[t]
  \centering
  \scalebox{0.85}{
    \tikzset{
      every node/.style={circle, draw,inner sep=1mm, minimum size=0.55cm, draw, thick, black, fill=gray!20, text=black},
      every path/.style={thick}
    }
    \begin{tikzpicture}[align=center,node distance=2cm]
      \node [] (1) {1};
      \node [] (2) [right of=1] {2};
      \node [] (3) [right of=2]    {3};
      \node [] (5) [above right of=3]    {5};
      \node [] (4) [below right of=5]    {4};

      \draw[blue] [-latex] (1) edge (2);
      \draw[blue] [-latex] (2) edge (3);
      \draw[blue] [-latex] (3) edge (4);
      \draw[blue] [-latex, bend right] (1) edge (3);
      \draw[red] [-latex] (5) edge (4);
      \draw[red] [-latex] (5) edge (3);
    \end{tikzpicture}
  }
  \caption{\label{fig:iv-dag} A DAG on five vertices.  Vertex 5 indexes a
    hidden variable.}
\end{figure}
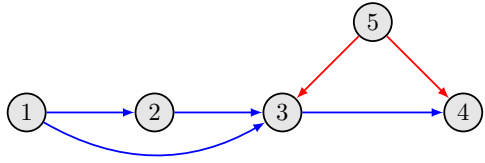

\begin{example}
  \label{ex:intro:iv}
  Let $G$ be the DAG in Figure~\ref{fig:iv-dag}, where vertex 5 indexes a
  hidden variable.  Then no conditional independence
  involving only the observed $X_1$, $X_2$, $X_3$, and $X_4$
  holds for all covariance matrices in $\mathcal{M}(G)$.  Instead, a
  positive definite $4\times 4$ matrix $\Sigma=(\sigma_{ij})$ is the
  projection of a matrix in $\mathcal{M}(G)$ if and only if
  \begin{align*}
    \left|\Sigma_{12, 34}\right|:= \det(\Sigma_{12, 34}) &\;=\;
                                 \sigma_{13}\sigma_{24}-\sigma_{14}\sigma_{23} \;=\;0
  \end{align*}
  and $\sigma_{j3}=0$ implies $\sigma_{j4}=0$ for $j=1,2$.  
\end{example}

In the example just given the key constraint is a determinant of the
covariance matrix that cannot be explained by $d$-separation. 
A major advance in this decade was the
introduction of trek separation, which is a graphical criterion that
can be used to decide the vanishing of any subdeterminant of the
covariance matrix \cite{DST,STD}.  Although more work is required to
fully exploit trek separation in model equivalence criteria, the
notion has already seen application in parameter identification
problems \cite{div}.

While greatly generalizing Gaussian conditional independence,
determinantal constraints are again not sufficient to describe the
sets $\mathcal{M}(G)$ after projection to the covariance matrix of
observed variables.  The following example is due to Thomas Verma.

\begin{figure}[t]
  \centering
  \scalebox{0.85}{
    \tikzset{
      every node/.style={circle, draw,inner sep=1mm, minimum size=0.55cm, draw, thick, black, fill=gray!20, text=black},
      every path/.style={thick}
    }
    \begin{tikzpicture}[align=center,node distance=2cm]
      \node [] (1) {1};
      \node [] (2) [right of=1] {2};
      \node [] (3) [right of=2]    {3};
      \node [] (5) [above right of=3]    {5};
      \node [] (4) [below right of=5]    {4};

      \draw[blue] [-latex] (1) edge (2);
      \draw[blue] [-latex] (2) edge (3);
      \draw[blue] [-latex] (3) edge (4);
      \draw[blue] [-latex, bend right] (1) edge (3);
      \draw[red] [-latex] (5) edge (4);
      \draw[red] [-latex] (5) edge (2);
    \end{tikzpicture}
  }
    \caption{\label{fig:verma-dag} The Verma graph. Vertex 5 indexes a
      hidden variable. }
\end{figure}
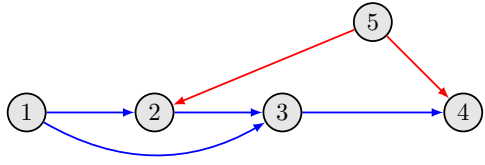

\begin{example}
  \label{ex:intro:verma}
  Let $G$ be the graph from Figure~\ref{fig:verma-dag}.
  Then as in the first example no conditional independence that holds
  for $\mathcal{M}(G)$ involves only the observed variables $X_1$,
  $X_2$, $X_3$, and $X_4$.  Instead, a positive definite $4\times 4$
  matrix $\Sigma=(\sigma_{ij})$ is the projection of a matrix in
  $\mathcal{M}(G)$ if and only if
  \begin{align}\label{eq:Verma_poly}
    f_{\text{Verma}} \;=\; \sigma_{11}\sigma_{13}\sigma_{22}\sigma_{34} - \sigma_{11}\sigma_{13}\sigma_{23}\sigma_{24} - \sigma_{11}\sigma_{14}\sigma_{22}\sigma_{33} + \sigma_{11}\sigma_{14}\sigma_{23}^2\\
    - \sigma_{12}^2\sigma_{13}\sigma_{34} + \sigma_{12}^2\sigma_{14}\sigma_{33} + \sigma_{12}\sigma_{13}^2\sigma_{24} - \sigma_{12}\sigma_{13}\sigma_{14}\sigma_{23} &\;=\; 0;\notag                                  
  \end{align}
  (compare Example 3.3.14 in \cite{loas}). The polynomial
  $f_{\text{Verma}}$ is not a subdeterminant of $\Sigma$ and,
  therefore, is neither explained by $d$-separation nor by
  trek-separation.
\end{example}

Another key advance in the area is a graph decomposition result of Jin
Tian \cite{tian:pearl:2002}; see also \cite[Sections 5-6]{D}.  This
result allows one to derive constraints by applying $d$-separation in
certain subgraphs.  In particular, the vanishing polynomial
$f_{\text{Verma}}$ from~\eqref{eq:Verma_poly} can be shown to arise
from the independence of variables $X_1$ and $X_4$ that holds for the
subgraph obtained by removing the edges $1\to 3$ and $2\to 3$ from the
Verma graph in Figure~\ref{fig:verma-dag}.  For further details, we
refer the reader to the review in \cite{EvansEtAl14}.

In the next example however, neither Tian's graph decomposition nor trek separation provide
any insight.

\begin{figure}[t]
  \centering
  \scalebox{0.85}{
    \tikzset{
      every node/.style={circle, draw,inner sep=1mm, minimum size=0.55cm, draw, thick, black, fill=gray!20, text=black},
      every path/.style={thick}
    }
    \begin{tikzpicture}[align=center,node distance=2cm]
      \node [] (c) {1};
      \node [] (ac) [above right of=c]    {5};
      \node [] (a) [below right of=ac] {2};
      \node [] (bc) [above right of=a]    {6};
      \node [] (b) [below right of=bc]    {3};
      \node [] (ad) [above right of=b]    {7};
      \node [] (d) [below right of=ad]    {4};

      \draw[blue] [-latex] (a) edge (b);
      \draw[blue] [-latex] (b) edge (d);
      \draw[red] [-latex] (ad) edge (a);
      \draw[red] [-latex] (ad) edge (d);

      \draw[red] [-latex] (ac) edge (a);
      \draw[red] [-latex] (ac) edge (c);
      
      \draw[red] [-latex] (bc) edge (b);
      \draw[red] [-latex] (bc) edge (c);
    \end{tikzpicture}
  }
    \caption{\label{fig:ommen-dag}  Graph based on 
      \cite[Fig.~1]{VanOmmenMooij_UAI_17}.  Vertices 5, 6 and 7 index
      hidden variables. }
\end{figure}

\begin{example}
  \label{ex:intro:ommen}
  Let $G$ be the graph from Figure~\ref{fig:ommen-dag}.  There are
  four observed variables, and projecting $\mathcal{M}(G)$ gives a set
  of codimension one.  As discussed in \cite{VanOmmenMooij_UAI_17},
  any covariance $\Sigma=(\sigma_{ij})\in\mathcal{M}(G)$ satisfies the
  constraint
  \begin{equation}
    \label{eq:ommen1}
    f_{\text{vOM}}\;=\;\sigma_{22}\sigma_{34}\sigma_{13}-
    \sigma_{22}\sigma_{33}\sigma_{14}-
    \sigma_{23}\sigma_{24}\sigma_{13}+\sigma_{23}^2\sigma_{14} 
    \;=\;0.
  \end{equation}
  The irreducible polynomial in~(\ref{eq:ommen1}) defines the
  hypersurface that contains the projection of $\mathcal{M}(G)$. 
\end{example}

A closer look at Examples~\ref{ex:intro:verma}
and~\ref{ex:intro:ommen} reveals some common structure.  Both
constraints are \emph{nested determinants}, by which we mean
determinants of a matrix whose entries are determinants themselves.
This observation is the point of departure for our paper.

\begin{example}
  \label{ex:intro:verma2:ommen2}
  The Verma polynomial from Example~\ref{ex:intro:verma} admits a
  compact representation through nested determinants, namely,
  \begin{align}
            \label{VermaMatrix3}
    f_{\text{Verma}}
    &\;=\; \begin{vmatrix}
      \left|\Sigma_{123,123}\right| & \left|\Sigma_{123,124}\right|\\
      \left|\Sigma_{1,3}\right| & \left|\Sigma_{1,4}\right|\\ 
    \end{vmatrix}.
  \end{align}
  Such a representation is generally not unique.  For instance,
  \begin{align}
    \label{VermaMatrix2}
    f_{\text{Verma}}&\;=\; \begin{vmatrix}
      \left|\Sigma_{123,134}\right| & \left|\Sigma_{123,234}\right|\\
      \left|\Sigma_{1,1}\right| & \left|\Sigma_{1,2}\right|\\ 
    \end{vmatrix}
%     \\
% \label{VermaMatrix1}
%     &
        \;=\;\begin{vmatrix}
          \left|\Sigma_{12, 12}\right| & \left|\Sigma_{12, 13}\right|\\
      \left|\Sigma_{34, 12}\right| & \left|\Sigma_{34, 13}\right|
    \end{vmatrix}.
  \end{align}
  The polynomial from Example~\ref{ex:intro:ommen} is also a nested
  determinant, namely,
  \begin{equation}
    \label{eq:ommen2}
      f_{\text{vOM}}\;=\;
    \begin{vmatrix}
      \left|\Sigma_{23,23} \right| & \left|\Sigma_{23,24} \right|\\
      \left|\Sigma_{1,3}\right| & \left|\Sigma_{1,4}\right|\\
    \end{vmatrix}.
  \end{equation}
  We are not aware of any literature emphasizing these types of representations.
\end{example}

In this paper, we investigate combinatorial conditions on the graph
$G$ that entail the vanishing of nested determinants.  We give a rigorous
definition of the models we study in Section~\ref{sec:background},
where we also provide background on the current knowledge of their
description.  In particular, we introduce mixed graph models that play
an important role in model selection \cite[Section 5.2]{drt17}.
Section~\ref{sec:ancestral} shows how nested determinants arise under
conditions of ancestrality.  In Theorem~\ref{thm:cut-ancestral} we show 
that such determinants completely describe the model $\mathcal M(G)$ 
for a wide class of mixed graphs that are (nearly) ancestral. Section~\ref{sec:restricted-treks}
describes our notion of {\em restricted trek separation} in the
setting of arbitrary acylic mixed graphs.  In
Section~\ref{sec:nested-determinants}, we show how the vanishing of
nested determinants can follow from restricted trek separation.  The
result we present also implies the vanishing of the constraints exhibited for (nearly)
ancestral graphs in Section~\ref{sec:ancestral}.  In
Section~\ref{sec:beyond-theorem}, we give examples that involve
recursive nesting of determinants and are beyond the scope of our
results.  Nevertheless, we can relate these examples back to restricted
trek separation.  While our focus is on acyclic mixed graphs, our last
example shows that a nested determinant may also arise for graphs
containing directed cycles.  We conclude with
Section~\ref{sec:discussion}, where we discuss future work and open
problems.
% and we
% show examples of more general hidden variable models in which the
% constraints are given by nested determinants as well.

%%% Local Variables:
%%% TeX-master: "nested_dets"
%%% End:

\section{Background}
\label{sec:background}

\subsection{Structural equation models}

Let $\epsilon=(\epsilon_i:i\in V)$ be the random error vector for the
equation system in~(\ref{eq:structural-general}).  As we are only
concerned with the covariance structure, we disregard the offsets
$\lambda_{0i}$.  The system can then be written as
\begin{equation}
  \label{eq:structural-equations-vector}
  X = \Lambda^TX + \epsilon,
\end{equation}
where the matrix $\Lambda = (\lambda_{ij})\in\mathbb R^{V\times V}$
holds the unknown coefficients.  Let
$\Omega = (\omega_{ij}) = \text{Var}[\epsilon]\in\mathbb R^{V\times
  V}$ be the covariance matrix of $\epsilon$, which we assume positive
definite.  Assuming further that $I - \Lambda$ is invertible, the
random vector $X= (I - \Lambda)^{-1}\epsilon$ is the unique solution
to the linear system in~(\ref{eq:structural-equations-vector}) and has
covariance matrix
\begin{equation}
  \text{Var}[X] \;=\; (I - \Lambda)^{-T}\Omega(I - \Lambda)^{-1}
  %=: \phi(\Lambda, \Omega)
  .\label{eq:VarX}
\end{equation}

In the introduction we focused on the case where the individual error
terms $\epsilon_i$ are independent.  Their covariance matrix $\Omega$
is then diagonal.  In this case, a model postulating that some of the
coefficients in $\Lambda$ are zero is conveniently represented by a
directed graph, as was our setup in Section~\ref{sec:introduction}.
Going forward, we also allow for dependence among the $\epsilon_i$ and
a possibly non-diagonal matrix $\Omega$.  Nonzero off-diagonal terms of
$\Omega$ are commonly represented by adding bidirected edges to the
considered directed graph.  %, which yields a mixed graph.

A {\em mixed graph} is a triple $G = (V, \mathcal{D}, \mathcal{B})$, where
$\mathcal{D}\subset V\times V$ is the set of {\em directed} edges, and $\mathcal{B}$ is
the set of {\em bidirected} edges which is comprised of unordered pairs
of elements of $V$.  We denote a directed edge from $i$ to $j$ by
$i\to j$, and a bidirected edge by $i\bi j$.  Let
$\mathbb R^\mathcal{D}$ be the set of $V\times V$ matrices $\Lambda$ with
support $\mathcal{D}$, i.e.,
\[
\mathbb R^\mathcal{D} = \{\Lambda\in\mathbb R^{V\times V}: \lambda_{ij} = 0
\text{ if } i\to j\not\in \mathcal{D}\}.
\]
Let $\mathbb R^\mathcal{D}_{\text{reg}}$ be the subset of matrices
$\Lambda\in\mathbb R^\mathcal{D}$ for which $I - \Lambda$ is invertible.  Let
$\mathit{PD}_V$ be the cone of positive definite $V\times V$ matrices, and
define $\mathit{PD}(\mathcal{B})$ to be the subcone of matrices supported over $\mathcal{B}$, i.e.,
\[
\mathit{PD}(\mathcal{B}) = \{\Omega= (\omega_{ij})\in \mathit{PD}_V: \omega_{ij} = 0 \text{ if }
i\neq j \text{ and } i\leftrightarrow j\not\in \mathcal{B}\}.
\]
The mixed graph $G$ is acyclic if its directed part $(V,\mathcal{D})$ does not
contain any directed cycles.  In this case, $V$ can be ordered such
that all matrices $\Lambda\in\mathbb R^\mathcal{D}$ are strictly
upper-triangular.  Thus, the determinant $|I-\Lambda|= 1$ and
$\mathbb R^\mathcal{D} = \mathbb R^\mathcal{D}_{\text{reg}}$. By Cramer's rule, the
covariances in $\text{Var}[X]$ in~(\ref{eq:VarX}) are then polynomial
functions of the entries of $\Lambda$ and $\Omega$.

Taking the error $\epsilon$ to be Gaussian, a given mixed graph
induces the following statistical model for the joint distribution of
$X$.

\begin{definition}\label{def:LSEM}
  The {\em linear structural equation model} given by a mixed graph
  $G = (V, \mathcal{D}, \mathcal{B})$ is the family of all multivariate normal
  distributions on $\mathbb R^V$ with covariance matrix in the set
$$\mathcal{M}(G) = \left\{(I - \Lambda)^{-T}\Omega(I - \Lambda)^{-1}:
  \Lambda \in \mathbb R^\mathcal{D}_{\text{reg}},\; \Omega \in \mathit{PD}(\mathcal{B})\right\}.$$ 
\end{definition}

The set $\mathbb R^\mathcal{D}_{\text{reg}}\times \mathit{PD}(\mathcal{B})$ is
semialgebraic.  Since $\mathcal{M}(G)$ is the image of this set under
a rational map, the Tarski-Seidenberg theorem yields that
$\mathcal{M}(G)$ itself is a semialgebraic set and, thus, admits a
polynomial description.  In this paper, we are interested in studying
polynomial equations that are satisfied by the matrices in
$\mathcal{M}(G)$.  With $\Sigma = (\sigma_{ij})$ interpreted as a
symmetric $V\times V$ matrix of indeterminates, define
$\mathbb R[\Sigma]$ to be the ring of polynomials in the
$\sigma_{ij}$.  Then the polynomial relations we seek to understand
make up the vanishing ideal
$$
\mathcal{I}(G):=\{f\in\mathbb R[\Sigma]: f(\Sigma) = 0 \text{ for all }
\Sigma\in\mathcal M(G)\}.
$$

Suppose a variable $X_j$, $j\in V$, is hidden.  Then the remaining
variables $(X_i:i\not=j)$ have their covariance matrix in the set
obtained by projecting each matrix in $\mathcal{M}(G)$ onto its
$(V\setminus\{j\})\times (V\setminus\{j\})$ submatrix.  Two comments
are in order.  First, we emphasize that for a fixed $j\in V$, the
polynomials $f\in\mathcal{I}(G)$ that do not involve any of the
indeterminates indexed by $j$, i.e., $f$ is free of $\sigma_{jk}$ for
$k\in V$, give precisely the polynomial constraints holding for the
model in which random variable $X_j$ is hidden.  Second, the paradigm
of mixed graphs allows one to directly capture relations after
projection.  Indeed, a graphical operation known as ``latent
projection'' creates a new mixed graph $G'$ over the observed
variables that represents key relations among covariances of observed
variables; see \cite[Section 2.6]{pearl:book},
\cite{koster:2002} or \cite{wermuth}.  For
instance, the three examples from our introduction would be
represented by the three mixed graphs in
Figure~\ref{fig:iv-verma-ommen-mixed}.  In these examples, the ideal
$\mathcal{I}(G')$ of the given mixed graph coincides with the ideal of
polynomial relations among the observed covariances in the hidden
variable model given by the original DAG $G$.

\begin{figure}[t]
  \centering
    \tikzset{
      every node/.style={circle, draw,inner sep=1mm, minimum size=0.55cm, draw, thick, black, fill=gray!20, text=black},
      every path/.style={thick}
    }
  \scalebox{0.85}{
    (a) \;
    \begin{tikzpicture}[align=center,node distance=2.2cm]
      \node [] (1) {1};
      \node [] (2) [right of=1] {2};
      \node [] (3) [right of=2]    {3};
      \node [] (4) [right of=3]    {4};
      
      \draw[blue] [-latex] (1) edge (2);
      \draw[blue] [-latex] (2) edge (3);
      \draw[blue] [-latex] (3) edge (4);
      \draw[blue] [-latex, bend right] (1) edge (3);
      \draw[red] [latex-latex, bend left] (3) edge (4);
    \end{tikzpicture}
    \qquad 
    (b) \;
    \begin{tikzpicture}[align=center,node distance=2.2cm]
      \node [] (1) {1};
      \node [] (2) [right of=1] {2};
      \node [] (3) [right of=2]    {3};
      \node [] (4) [right of=3]    {4};
      
      \draw[blue] [-latex] (1) edge (2);
      \draw[blue] [-latex] (2) edge (3);
      \draw[blue] [-latex] (3) edge (4);
      \draw[blue] [-latex, bend right] (1) edge (3);
      \draw[red] [latex-latex, bend left] (2) edge (4);
    \end{tikzpicture}
  }
  \vspace{.3cm}
  
    \scalebox{0.85}{
    (c) \;
    \begin{tikzpicture}[align=center,node distance=2.2cm]
      \node [] (1) {1};
      \node [] (2) [right of=1] {2};
      \node [] (3) [right of=2]    {3};
      \node [] (4) [right of=3]    {4};
      
      \draw[blue] [-latex] (2) edge (3);
      \draw[blue] [-latex] (3) edge (4);
      \draw[red] [latex-latex, bend left] (1) edge (3);
      \draw[red] [latex-latex] (1) edge (2);
      \draw[red] [latex-latex, bend left] (2) edge (4);
      %% just for figure placement
      \draw[white] [latex-latex, bend right] (1) edge (2);
    \end{tikzpicture}
  }
  \caption{(a)-(c) Mixed graphs obtained by latent projection of the
    DAGs in Figures~\ref{fig:iv-dag}-\ref{fig:ommen-dag}, respectively.}\label{fig:iv-verma-ommen-mixed}
\end{figure}
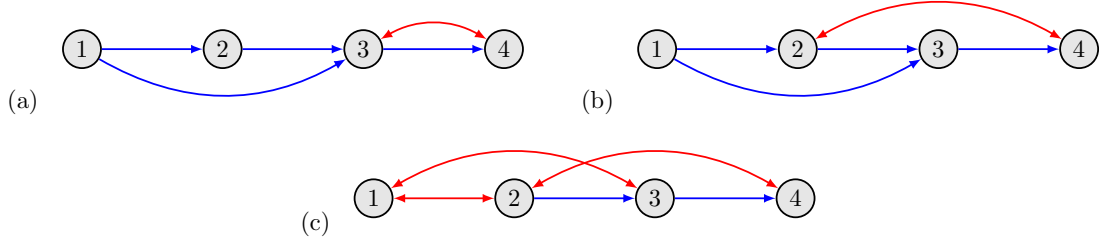

\subsection{Trek rule}
\label{sec:trek-rule}

Again let $G=(V,\mathcal{D},\mathcal{B})$ be any mixed graph, possibly cyclic.  The starting point
for any combinatorial understanding of polynomials in the vanishing
ideal $\mathcal{I}(G)$ is the \emph{trek rule}.  This rule specifies
each entry of the covariance matrix in~(\ref{eq:VarX}) as a sum of
monomials associated with certain paths in the graph.

\begin{definition}
  \label{def:trek}
  A {\em trek} is a path $\tau$ of the form
  \begin{enumerate}
  \item[(a)]
    $i_\ell \leftarrow \cdots \leftarrow i_1 \leftrightarrow j_1
    \rightarrow \cdots \rightarrow j_r$, or
  \item[(b)]
    $i_\ell \leftarrow \cdots \leftarrow i_1 = j_1
    \rightarrow \cdots \rightarrow j_r$,
  \end{enumerate}
  for integers $\ell, r \geq 0$ with $\ell+r\ge 1$.  Here, a path may
  visit a vertex more than once.  If $\ell=0$, the trek is simply the
  directed path $j_1\to \dots\to j_r$.  Similarly, it is
  $i_\ell\leftarrow\dots\leftarrow i_1$ if $r=0$.  We call $\tau$ a trek from $i_\ell$ to $j_r$ or also a
  trek between $i_\ell$ and $j_r$.  The sets $\{i_k:k=1,\dots,\ell\}$
  and $\{j_k:k=1,\dots,r\}$ are the \emph{left side} and the
  \emph{right side} of $\tau$, respectively.
\end{definition} 

Let $\Lambda=(\lambda_{ij})\in\mathbb{R}^\mathcal{D}_{\text{reg}}$ and
$\Omega=(\omega_{ij})\in\mathit{PD}(\mathcal{B})$.  To any trek $\tau$, specified
as in Definition~\ref{def:trek}, associate a {\em trek monomial}
\begin{equation}\label{eq:trek-monomial}
  \sigma(\tau) \;=\; \omega_{i_1j_1}\prod_{k=1}^{\ell-1}
  \lambda_{i_ki_{k+1}}\prod_{k=1}^{r-1} \lambda_{j_kj_{k+1}}. 
\end{equation}
The \emph{trek rule} now states that the covariance
matrix $\Sigma=(\sigma_{ij})=(I-\Lambda)^{-T}\Omega(I-\Lambda)^{-1}$
has its entries given by
\begin{equation}
  \sigma_{ij}\;=\; \sum_{\text{treks }\tau \text{ from }i \text{ to }
    j} \sigma(\tau).\label{eq:trek-rule}
\end{equation}
The rule, which originates in the work of \cite{wright:1934}, is
obtained by observing that
$(I-\Lambda)^{-1}=I+\Lambda+\Lambda^2+\dots$.  The right-hand side
of~\eqref{eq:trek-rule} is a polynomial when $G$ is acyclic and a
(formal) power series otherwise.

\subsection{Conditional independence and subdeterminants}

The notion of $d$-separation allows one to decide by inspection of
paths in a mixed graph $G$ whether a conditional independence relation
holds for all distributions in the model given by $G$; see
e.g.~\cite[Section 10]{D}.  In algebraic terms, for a Gaussian joint
distribution, variables $X_i$ and $X_j$ are conditionally
independent given a subvector $X_S$ with $i,j\not\in S$ if and only if
the subdeterminant $|\Sigma_{iS, jS}|$ is zero.  Here, $iS$ denotes
the union of a singleton set $\{i\}$ and the set $S$.  Thus,
$d$-separation gives a combinatorial characterization of when a
subdeterminant of the form $|\Sigma_{iS, jS}|$ belongs to the ideal
$\mathcal{I}(G)$.  If $G$ is a DAG, then the covariance model
$\mathcal{M}(G)$ admits a semi-algebraic description by conditional
independence.  Indeed, $\mathcal{M}(G)$ is the set of positive
definite matrices $\Sigma$ for which all conditional independence
determinants $|\Sigma_{iS, jS}|$ associated with the graph $G$ vanish.
This is also true for mixed graphs that are maximal ancestral
\cite{richardson:2002}, but false more generally as the examples in
the introduction show.

In seminal work, Sullivant, Talaska and Draisma \cite{STD} move beyond
conditional independence determinants and give a combinatorial
characterization of when an arbitrary subdeterminant $|\Sigma_{A,B}|$ is
in $\mathcal I(G)$.  We briefly review their concept of
trek-separation; see also \cite[Section 11]{D}.

\begin{definition}
  \label{def:trek-sep}
  Two sets $A, B\subseteq V$ are {\em trek-separated} by the pair
  $(S_L,S_R)$, where $S_L, S_R\subseteq V$, if every trek between a
  vertex from $A$ and a vertex from $B$ intersects either $S_L$ on its
  left side or $S_R$ on its right side.
\end{definition}

In the case $|A|=|B|=m$ the following theorem shows that $|\Sigma_{A,B}|\in\mathcal I(G)$
if and only if the sets of vertices $A$ and $B$ are {\em
  trek-separated} by a pair $(S_L,S_R)$ with $|S_L| + |S_R|<|A|$.

\begin{theorem}[Thm.~2.17,~\cite{STD}]\label{thm:trek_separation}
  Let $A,B\subseteq V$.  The submatrix $\Sigma_{A, B}$ has rank at
  most $r$ for all covariance matrices $\Sigma\in\mathcal{M}(G)$ if
  and only if there exist subsets $S_L, S_R\subseteq V$ such that
  $|S_L| + |S_R| \leq r$ and $(S_L, S_R)$ trek-separates $A$ from $B$.
  For a generic choice of $\Sigma\in\mathcal{M}(G)$,
  $$
  \text{rank}(\Sigma_{A,B}) \;=\; \min\{|S_L| + |S_R|: (S_L, S_R)\
  \text{trek-separates } A \text{ from } B\}.$$ 
\end{theorem}

While trek-separation greatly generalizes $d$-separation and can yield
a generating set of $\mathcal{I}(G)$ for some mixed graphs
\cite{MR3566228}, it is in general not sufficient to understand the vanishing
ideal $\mathcal I(G)$ as we demonstrated in
Examples~\ref{ex:intro:verma} and~\ref{ex:intro:ommen}.

%%% Local Variables:
%%% TeX-master: "nested_dets"
%%% End:

\section{Ancestral vertices and overdetermined linear systems}
\label{sec:ancestral}

We now proceed to a first combinatorial condition (see Proposition~\ref{prop:ancestral-nested-det}) for the vanishing of
very special nested determinants (Definition~\ref{def:parentally-nested}).
% In section~\ref{sec:nested-determinants} we show
% that this characterization is a special case a combinatorial criterion
% that we call {\em restricted trek separation}.
Fix a mixed graph $G=(V,\mathcal{D},\mathcal{B})$, and let $\Sigma\in\mathit{PD}_V$.  For
a pair of matrices $\Lambda\in\mathbb{R}^\mathcal{D}_{\text{reg}}$ and
$\Omega\in\mathit{PD}(\mathcal{B})$, it holds that
\[
  \Sigma =(I-\Lambda)^{-T}\Omega(I-\Lambda)^{-1}
  \iff
  (I-\Lambda)^T\Sigma(I-\Lambda) =\Omega.
\]
In turn, $\Sigma\in\mathcal M(G)$ if and only if
\begin{align}
  \label{eq:I-L.Sigma.I-L}
  \left[(I-\Lambda)^T\Sigma(I-\Lambda)\right]_{ij} \;=\;0 \quad\forall
    i,j\in V \text{ with } i\not=j,\; i\bi j\notin \mathcal{B}.
\end{align}
For some graphs it is known that all entries of $\Lambda$ can be
recovered as rational expressions of $\Sigma$, at least for generic
choices of positive definite $\Sigma$.  For instance, the half-trek
criterion \cite{foygel:draisma:drton:2012} and its extensions
\cite{NIPS2016_6223,drton:weihs:2015,div} can be used to certify
graphically that such \emph{rational identification} of $\Lambda$ from
$\Sigma$ is possible and to find rational expressions.  If now both
the $i$th and the $j$th column of $\Lambda$ are rationally identifiable
from $\Sigma$, then the left-hand side of the equation
in~(\ref{eq:I-L.Sigma.I-L}) can be expressed as a rational function of
$\Sigma$.  If $i\not=j$ and $i\bi j\notin \mathcal{B}$, then one finds a
rational constraint on $\Sigma$ that after clearing denominators
yields a polynomial in $\mathcal{I}(G)$.  This approach is used, for
instance, in \cite{VanOmmenMooij_UAI_17}.

In this section we follow a similar approach in which we substitute
solutions for some of the entries of $\Lambda$ that appear
in~\eqref{eq:I-L.Sigma.I-L}.  However, we only linearize the equations
and then observe that nested determinantal constraints arise from
overdetermined linear equation systems.  Specifically, we study the
following constraints.

\begin{definition}
  \label{def:parentally-nested}
  Let $i$ be a vertex of the mixed graph $G$, and let $J$ be a subset
  of vertices in $G$.  Define a matrix of polynomials of size
  $(|\pa(i)|+|J|)\times (|\pa(i)|+1)$ as
  \begin{equation}
    \label{eq:nested-det-matrix}
    F_{i,J} \;=\;
    \left(\left|\Sigma_{\pa(r)\uplus
            \{r\},\pa(r)\uplus\{c\}}\right| \right)_{r\in \pa(i)\uplus
        J,
        c\in\pa(i)\uplus\{i\}}.    
  \end{equation}
  The \emph{parentally nested determinants} for the pair $(i,J)$ are
  the minors of order $|\pa(i)|+1$ of the matrix $F_{i,J}$. When
  $J=\{j\}$ is a singleton, there is only one parentally nested
  determinant
  \begin{equation}
    \label{eq:nested-det-parent}
    f_{ij} \;=\; \left| \left(\left|\Sigma_{\pa(r)\uplus
            \{r\},\pa(r)\uplus\{c\}}\right| \right)_{r\in \pa(i)\uplus\{j\},
        c\in\pa(i)\uplus\{i\}}\right|.
  \end{equation}
\end{definition}

Here, index sets are treated as multisets with possibly repeated
elements, and the determinants are formed according to a prespecified
linear order for the vertex set $V$.  The symbol $\uplus$ stands for
the sum (or disjoint union) of multisets; e.g.,
$\{1,1,2\}\uplus\{1,3\}=\{1,1,1,2,3\}$.

Suppose $j\in J\cap\pa(i)$.  Then $j$ is repeated in the row index set
$\pa(i)\uplus J$ for the matrix $F_{i,J}$.  In this case $j$ indexes two
rows for a minor, which is then zero.  In particular, if $j\in\pa(i)$
then $f_{ij}=0$.  We may therefore always restrict the set $J$ to
satisfy $J\cap\pa(i)=\emptyset$.

A repeated index may also arise for the column index sets of the
matrices whose determinants yield the entries of $F_{i,J}$.  Indeed,
if $c\in\pa(i)\cup\{i\}$ is also in $\pa(r)$ for
$r\in \pa(i)\cup\{j\}$, then the $(r,c)$ entry of $F_{i,J}$ is zero.

\begin{example}
  It holds that $f_{\text{Verma}}=f_{41}$ in
  Example~\ref{ex:intro:verma}, and $f_{\text{vOM}}=f_{41}$ in
  Example~\ref{ex:intro:ommen}.
\end{example}

In the remainder of this section, we identify conditions under which
parentally nested determinants vanish.

\begin{definition}
  A vertex $j$ in the mixed graph $G=(V,\mathcal{D},\mathcal{B})$ is \emph{ancestral} if
  (i) $j$ is not on any directed cycle, and (ii) no vertex $k\not=j$
  has both $k\bi j\in \mathcal{B}$ and a directed path from $k$ to $j$.
\end{definition}

Let $\Lambda\in\mathbb{R}^\mathcal{D}_{\text{reg}}$ and
$\Omega\in\mathit{PD}(\mathcal{B})$, and define
$\Sigma=(I-\Lambda)^{-T}\Omega(I-\Lambda)^{-1}\in\mathcal{M}(G)$.  If
$j$ is ancestral, then all treks from a vertex $r\in\pa(j)$ to $j$ end
with a directed edge pointing to $j$.   The
trek rule from~(\ref{eq:trek-rule}) then implies that
\begin{equation}
  \label{eq:ancestral}
  \Sigma_{\pa(j),\pa(j)} \Lambda_{\pa(j),j} \;=\;
  \Sigma_{\pa(j),j}.
\end{equation}
For our next result it is convenient to introduce the set of \emph{siblings} of a vertex
$j$, which is $\sib(j)=\{k\in V: k\bi i\in \mathcal{B}\}$, the set of neighbors
of $j$ in the bidirected part of the graph.

\begin{proposition}
  \label{prop:ancestral-nested-det}
  Let $i$ be a vertex of a mixed graph $G=(V,\mathcal{D},\mathcal{B})$
  such that
   \begin{enumerate}%[label=(\roman*)]
  \item[(i)] $\pa(i)\cap\sib(i)=\emptyset$,
  \item[(ii)] all vertices in $\pa(i)$ are ancestral, and
  \item[(iii)] the set $J$ of all ancestral vertices in $V\setminus(\pa(i)\cup\sib(i)\cup\{i\})$ is non-empty.
  \end{enumerate}
  Then the parentally nested determinants for the pair $(i,J)$ are
  in the vanishing ideal $\mathcal{I}(G)$.
  %
  % Let $J$ be the set of nodes $j\notin \pa(i)\cup\{i\}$ that are ancestral
  % and have $j\bi i\notin \mathcal{B}$.  Then the parentally nested determinants
  % for the pair $(i,J)$ are all in the vanishing ideal
  % $\mathcal{I}(G)$.
\end{proposition}
\begin{proof}
  Let $\Lambda=(\lambda_{ab})\in\mathbb{R}^\mathcal{D}_{\text{reg}}$ and
  $\Omega\in\mathit{PD}(\mathcal{B})$, and define
  $\Sigma=(I-\Lambda)^{-T}\Omega(I-\Lambda)^{-1}\in\mathcal{M}(G)$.
  %Let $J=V\setminus(\pa(i)\cup\sib(i)\cup\{i\})$.  
  Neither $\pa(i)$ nor $J$ contains vertices in $\sib(i)$.
  Fixing $r\in pa(i)\cup J$, (\ref{eq:I-L.Sigma.I-L}) implies that
  \begin{equation}
    \label{eq:1}
    \left[(I-\Lambda)^T\Sigma(I-\Lambda) \right]_{ri} \;=\; 0.
  \end{equation}
  This equation becomes
  \begin{equation}
    \label{eq:2}
    \sigma_{ri} - \Lambda_{\pa(r),r}^T
    \Sigma_{\pa(r),i}-\Sigma_{r,\pa(i)}\Lambda_{\pa(i),i}+\Lambda_{\pa(r),r}^T\Sigma_{\pa(r),\pa(i)}\Lambda_{\pa(i),i} \;=\;0.
  \end{equation}
  Since all vertices in $\pa(i)\cup J$ are ancestral, we may
  use~(\ref{eq:ancestral}) to get the rational equation
\begin{multline}
  \label{eq:3}
  \sigma_{ri} - \Sigma_{r,\pa(r)}\Sigma_{\pa(r),\pa(r)}^{-1}
  \Sigma_{\pa(r),i}-\left(\Sigma_{r,\pa(i)}-
  \Sigma_{r,\pa(r)}\Sigma_{\pa(r),\pa(r)}^{-1}\Sigma_{\pa(r),\pa(i)}\right)\Lambda_{\pa(i),i} \;=\;0.
\end{multline}
Now observe that for any vertex $c$,
\begin{align}
  \label{eq:schur}
  \left(\sigma_{rc} - \Sigma_{r,\pa(r)}\Sigma_{\pa(r),\pa(r)}^{-1}
  \Sigma_{\pa(r),c} \right) \left|\Sigma_{\pa(r),\pa(r)}\right|
  &\;=\; \left|\Sigma_{\pa(r)\cup \{r\},\pa(r)\cup\{c\}}\right|.
\end{align}
Hence, multiplying the equation in~(\ref{eq:3}) by
$\left|\Sigma_{\pa(r),\pa(r)}\right|$ gives 
\begin{align}
  \label{eq:4}
  \left|\Sigma_{\pa(r)\cup \{r\},\pa(r)\cup\{i\}}\right| - \sum_{c\in\pa(i)}
  \left|\Sigma_{\pa(r)\cup \{r\},\pa(r)\cup\{c\}}\right|\cdot \lambda_{ci}
  &\;=\;0.
\end{align}
With one equation for every $r\in\pa(i)\cup J$, the system is
overdetermined and admits a solution only if the matrix $F_{i,J}$ from Definition~\ref{def:parentally-nested} has
rank at most $|\pa(i)|$.  This in turn implies the vanishing of its
minors. Note that in the case that $i$ is not trek reachable from $r$, the last equation is trivial and corresponds to a row of zeros in $F_{i,J}$.
\end{proof}

\begin{figure}[t]
  \centering
    \tikzset{
      every node/.style={circle, draw,inner sep=1mm, minimum size=0.55cm, draw, thick, black, fill=gray!20, text=black},
      every path/.style={thick}
    }
    \scalebox{0.85}{
    \begin{tikzpicture}[align=center,node distance=2.2cm]
      \node [] (1) {1};
      \node [] (2) [right of=1] {2};
      \node [] (3) [right of=2]    {3};
      \node [] (4) [right of=3]    {4};
      
      \draw[blue] [-latex] (1) edge (2);
      \draw[blue] [-latex] (2) edge (3);
      \draw[blue] [-latex, bend left] (2) edge (4);
      \draw[blue] [-latex] (3) edge (4);
    \end{tikzpicture}
  }
  \caption{DAG on 4 vertices used to illustrate the nested determinants $f_{ij}$.}\label{fig:dag-4}
\end{figure}
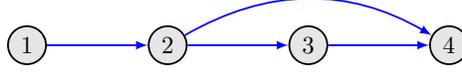

\begin{example}
  The graph $G$ from Figure~\ref{fig:dag-4} is a DAG, and thus all
  its vertices are ancestral.  As there are no bidirected edges,
  $f_{ij}\in\mathcal{I}(G)$ for all $i\not=j$.  As previously noted,
  for any graph $f_{ij}=0$ if $j\in\pa(i)$.  Here,
  $f_{21}=f_{32}=f_{42}=f_{43}=0$.  Moreover, $f_{12}=f_{34}=0$.  The
  nonzero polynomials are
  \begin{align*}
    f_{13}&=\left|\Sigma_{12,23}\right|,
          & f_{31}&=\sigma_{11}\cdot\left|\Sigma_{12,23}\right|,
    & f_{23}&=-\sigma_{12}\cdot\left|\Sigma_{12,23}\right|\\
    f_{14}&=\left|\Sigma_{123,234}\right|,
          &
            f_{41}&=\sigma_{11}\sigma_{22}\cdot\left|\Sigma_{123,234}\right|, 
    & f_{24}&=-\sigma_{12}\cdot\left|\Sigma_{123,234}\right|.    
  \end{align*}
  The irreducible polynomial $f_{13}$ corresponds to conditional
  independence of $X_1$ and $X_3$ given $X_2$.  The second irreducible
  polynomial $f_{14}$ encodes conditional independence of $X_1$ and
  $X_4$ given $(X_2,X_3)$.  It turns out that
  \begin{align*}
    \mathcal{M}(G)
    &\;=\;\{\Sigma\in\mathit{PD}_{\{1,2,3,4\}}:
      f_{13}(\Sigma)=f_{14}(\Sigma)=0\}\\
    &\;=\;\{\Sigma\in\mathit{PD}_{\{1,2,3,4\}}:
      f_{31}(\Sigma)=f_{41}(\Sigma)=0\}\\
    &\;=\;\{\Sigma\in\mathit{PD}_{\{1,2,3,4\}}:
      f_{ij}(\Sigma)=0 \;\forall i\not=j\}.
  \end{align*}
  In fact, the ideal
  $\langle f_{ij}:i\not=j\rangle=\langle f_{13},f_{14}\rangle$ differs
  from $\mathcal{I}(G)$ only through components that do not vanish at
  positive definite matrices \cite[Example 2]{roozbehani}.  Specifically,
  \[
   \langle f_{13},f_{14}\rangle \quad=\quad \mathcal{I}(G) \;\cap\;  \langle
   \left|\Sigma_{23,23}\right|,    \left|\Sigma_{13,23}\right|,
   \left|\Sigma_{12,23}\right|\rangle  .
 \]
 Here, 
 \[
   \mathcal{I}(G) \quad=\quad \langle  \left| \Sigma_{12,23}\right|,
   \left| \Sigma_{12,24}\right| ,\left| \Sigma_{12,34}\right|\rangle
 \]
 is generated by three subdeterminants, two of which are conditional
 independences.
\end{example}

\begin{figure}[t]
\centering
  \scalebox{0.85}{
\tikzset{
	every node/.style={circle, draw,inner sep=1mm, minimum size=0.55cm, draw, thick, black, fill=gray!20, text=black},
	every path/.style={thick}
}
\begin{tikzpicture}[align=center,node distance=2.2cm]
		\node [] (1) {1};
	\node [] (3) [below right of=1] {3};
	\node [] (2) [above right of=3] {2};
	\node [] (5) [below right of=2] {5};
	\node [] (4) [above right of=5] {4};

	\draw[blue] [-latex] (1) edge (3);
	\draw[blue] [-latex] (2) edge (3);
	\draw[blue] [-latex] (2) edge (5);
	\draw[blue] [-latex] (4) edge (5);
	\draw[red] [latex-latex, bend left] (1) edge (4);
	\draw[red] [latex-latex] (1) edge (5);
	\draw[red] [latex-latex] (1) edge (2);
	\draw[red] [latex-latex] (2) edge (4);
	\draw[red] [latex-latex] (3) edge (4);
      \end{tikzpicture}
      }
\caption{An ancestral graph that is not maximal.}\label{fig:ancestral-notMAG}
\end{figure}
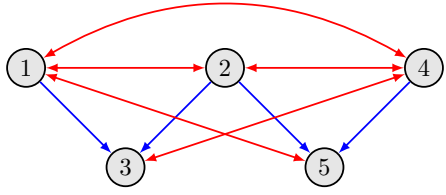

%\subsection{Ancestral graph with intersecting parent sets}
\begin{example}
  The mixed graph in Figure~\ref{fig:ancestral-notMAG} is an ancestral
  graph, that is, all vertices are ancestral \cite{richardson:2002}.  It
  is not maximal, that is, there are non-adjacent vertices,
  namely, $3$ and $5$, that cannot be $d$-separated.  There is then no
  conditional independence constraint associated to the non-adjacency.
  Precisely two of the $f_{ij}$ are nonzero, namely,
  \begin{align*}
    f_{35}&\;=\;
  \begin{vmatrix}
    |\Sigma_{1,1}| & |\Sigma_{1,2}| & |\Sigma_{1,3}|\\ 
    |\Sigma_{1,2}| & |\Sigma_{2,2}| & |\Sigma_{2,3}|\\
    |\Sigma_{124,245}| & |\Sigma_{224,245}| & |\Sigma_{234,245}|
  \end{vmatrix} \;=\;
  \begin{vmatrix}
    |\Sigma_{1,1}| & |\Sigma_{1,2}| & |\Sigma_{1,3}|\\ 
    |\Sigma_{1,2}| & |\Sigma_{2,2}| & |\Sigma_{2,3}|\\
    |\Sigma_{124,245}| & 0 & |\Sigma_{234,245}|
  \end{vmatrix},\\
    f_{53}&\;=\;
  \begin{vmatrix}
    |\Sigma_{2,2}| & |\Sigma_{2,4}| & |\Sigma_{2,5}|\\ 
    |\Sigma_{4,2}| & |\Sigma_{4,4}| & |\Sigma_{4,5}|\\
    |\Sigma_{123,122}| & |\Sigma_{123,124}| & |\Sigma_{123,125}|
  \end{vmatrix} \;=\;
  \begin{vmatrix}
    |\Sigma_{2,2}| & |\Sigma_{2,4}| & |\Sigma_{2,5}|\\ 
    |\Sigma_{4,2}| & |\Sigma_{4,4}| & |\Sigma_{4,5}|\\
    0 & |\Sigma_{123,124}| & |\Sigma_{123,125}|
  \end{vmatrix}.
  \end{align*}
  In fact, $f_{35}=f_{53}$, and $\mathcal{I}(G)=\langle
  f_{35}\rangle=\langle f_{53}\rangle$.  We note that there is also
  the alternative representation of 
  \begin{equation*}
    f_{35} \;=\;
      \begin{vmatrix}
      |\Sigma_{12,12}| & |\Sigma_{12,23}|\\ 
      |\Sigma_{124,245}| & |\Sigma_{234,245}|
    \end{vmatrix} 
    \;=\;
      \begin{vmatrix}
      |\Sigma_{24,24}| & |\Sigma_{24,25}|\\ 
      |\Sigma_{123,124}| & |\Sigma_{123,125}|
    \end{vmatrix}\; =\;
    f_{53}.
  \end{equation*}
%Here, we are looking at treks between $\{4,3\}$ and $\{4,5\}$ restricted to $V_L=\{3,4,5\}$ and $V_R=\{3,4,5\}$. All such treks go through 4 on the right.
\end{example}

\medskip

We now give a model description for a class of graphs that includes
all ancestral graphs.  It also covers the two graphs from
Figure~\ref{fig:iv-verma-ommen-mixed}(b)(c).  Recall that a
\emph{sink} of a mixed graph $G=(V,\mathcal{D},\mathcal{B})$ is any
vertex that is not a parent of any other vertex.  A subgraph of $G$ is
a mixed graph $G'=(V',\mathcal{D}',\mathcal{B}')$ with
$V'\subseteq V$, $\mathcal{D}'\subseteq \mathcal{D}$, and
$\mathcal{B}'\subseteq \mathcal{B}$.  We call $G$ \emph{globally
  identifiable} if $G$ is acyclic and none of its subgraphs
$G'=(V',\mathcal{D}',\mathcal{B}')$ has both a connected bidirected
part $(V',\mathcal{B}')$ and a unique sink vertex in its directed part
$(V',\mathcal{D}')$.  

For any set of
polynomials $\mathcal{F}\subset\mathbb{R}[\Sigma]$, we let
$\mathcal{V_F}=\{\Sigma:f(\Sigma)=0 \;\forall f\in\mathcal{F}\}$ be
the algebraic subset it defines in the space of symmetric matrices.

\begin{theorem}
  \label{thm:cut-ancestral}
  Let $G=(V,\mathcal{D},\mathcal{B})$ be a globally identifiable mixed graph with vertex
  set $V=[p]\equiv\{1,\dots,p\}$ enumerated in a topological order.
  Suppose all vertices in $[p-1]$ are ancestral.  Let $\mathcal{F}(G)$
  be the set of all parentally nested determinants obtained from the
  pairs $(i,[i-1]\setminus (\pa(i)\cup\sib(i)))$ for $i\in V$.  Then
  \[
    \mathcal{M}(G)\;=\; \mathit{PD}_V\cap\mathcal{V}_{\mathcal{F}(G)}.
  \]
\end{theorem}
\begin{proof}
  The inclusion
  $\mathcal{M}(G)\subseteq
  \mathit{PD}_V\cap\mathcal{V}_{\mathcal{F}(G)}$ follows from
  application of Proposition~\ref{prop:ancestral-nested-det}.
  
  To show the reverse inclusion, we proceed by induction on the number
  of vertices $p$.  The statement is trivial for $p=1$.  In the
  induction step, let
  $\Sigma\in\mathit{PD}_V\cap\mathcal{V}_{\mathcal{F}(G)}$.  Let
  $\Sigma_{[p-1],[p-1]}$ be the submatrix obtained by removing the
  $p$th row and column.  Let $G[p-1]=([p-1],\mathcal{D}[p-1],\mathcal{B}[p-1])$ be the
  subgraph induced by $[p-1]$.  
  Now,
  $\Sigma_{[p-1],[p-1]}\in\mathit{PD}_{[p-1]}\cap\mathcal{V}_{\mathcal{F}(G[p-1])}$.
  The induction hypothesis yields that
  $\Sigma_{[p-1],[p-1]}\in\mathcal{M}(G[p-1])$.  Let
  $\Sigma_{[p-1],[p-1]}=(I- \Lambda')^{-T}\Omega'(I-\Lambda')^{-1}$ for
  $\Lambda'\in\mathbb{R}^{\mathcal{D}[p-1]}$ and $\Omega'\in\mathit{PD}(\mathcal{B}[p-1])$.

  Consider the matrix $F_{p,[p-1]\setminus(\pa(p)\cup \sib(p))}$ from Definition~\ref{def:parentally-nested}
  evaluated at the given matrix $\Sigma$.  For each
  $r\in[p-1]\setminus\sib(p)$, divide the corresponding row by
  $|\Sigma_{\pa(r),\pa(r)}|>0$.  The resulting matrix $\bar F$ has
  entries
  \begin{equation}
    \sigma_{rc} - \Sigma_{r,\pa(r)}\Sigma_{\pa(r),\pa(r)}^{-1}
    \Sigma_{\pa(r),c} 
  \end{equation}
  for $r\in[p-1]\setminus\sib(p)$ and $c\in\pa(p)\cup\{p\}$;
  recall~\eqref{eq:schur}.  Using~\eqref{eq:ancestral}, we obtain that
  \begin{equation}
    \bar F \;=\; \left[ (I- \Lambda')^T\Sigma_{[p-1],V}
    \right]_{[p-1]\setminus\sib(p),\pa(p)\cup\{p\}}.
  \end{equation}
  Form the submatrix $\bar F_{[p-1]\setminus\sib(p),\pa(p)}$, that is,
  we omit the column indexed by $p$.  Then
  \begin{align*}
    \bar F_{[p-1]\setminus\sib(p),\pa(p)}
    &\;=\; \left[
      \Omega'(I-\Lambda')^{-1}
    \right]_{[p-1]\setminus\sib(p),\pa(p)}.
  \end{align*}
  Lemma 2 in \cite{MR2816341} yields that
  $\bar F_{[p-1]\setminus\sib(p),\pa(p)}$ has full column rank.

  Since $\Sigma\in\mathcal{V}_{\mathcal{F}(G)}$, the matrix
  $F_{p,[p-1]\setminus(\pa(p)\cup \sib(p))}$ and thus also $\bar F$
  do not have full column rank.  We conclude that the kernel of
  $\bar F$ contains a vector $x\in\mathbb{R}^{\pa(p)\cup\{p\}}$ for
  which the last coordinate $x_p\not=0$.  Dividing $x_{\pa(p)}$ by
  $-x_p$ gives a vector $\lambda_{\pa(p),p}\in\mathbb{R}^{\pa(p)}$
  that solves the equation system in~\eqref{eq:4}.  Define a
  $p\times p$ matrix $\Lambda\in\mathbb{R}^\mathcal{D}$ by using
  $\lambda_{\pa(p),p}$ to define its last column.  Then $\Lambda$
  solves~\eqref{eq:1} for $i=p$ and all $r\in[p-1]\setminus\sib(p)$
  and thus also~\eqref{eq:I-L.Sigma.I-L}.  Therefore,
  $\Sigma\in\mathcal{M}(G)$.
\end{proof}

The above facts leverage existence of ancestral vertices.  In the next
sections, we seek to give a more general condition for the vanishing of
nested determinants.  The results on vanishing nested determinants
from this section can be recovered as a special case; see
Proposition~\ref{prop:restricted_trek_ancestral}.

\section{Restricted Trek Separation}
\label{sec:restricted-treks}

As we reviewed in Section~\ref{def:trek-sep}, the notion of trek
separation \cite{STD} provides a combinatorial characterization of
when a subdeterminant of the covariance matrix $\Sigma$ vanishes over
a model $\mathcal{M}(G)$.  Underlying the trek separation result we
stated in Theorem~\ref{thm:trek_separation} is the observation that
determinants correspond to sums of certain products of trek monomials.
In this section, we recall this observation and then introduce a
notion of {\em restricted trek separation}, in which separation only
needs to occur with respect to treks that avoid certain vertices on
their left or right sides.  This notion will be used in
Section~\ref{sec:nested-determinants} to obtain conditions that imply
the vanishing of nested determinants.

Let $A$ and $B$ be two subsets of the vertex set of a mixed graph $G$,
with $|A| = |B|$.  A \emph{system of treks} from $A$ to $B$ is a set
of treks that each are between a vertex in $A$ and a vertex in $B$.  Let
$\mathcal{T}$ be such a system.  Then $\mathcal{T}$ has \emph{no sided intersection}
if any two distinct treks in $\mathcal{T}$ have disjoint left sides and
disjoint right sides.  In particular, each vertex in $A$ and each vertex
in $B$ is on precisely one trek, so that $\mathcal{T}$ induces a
bijection between $A$ and $B$.  Fixing an ordering of the elements of
$A$ and $B$, the trek system induces a permutation of $B$ in which
the $i$th element of $B$ is mapped to the end point of the trek that
starts at the $i$th element of $A$.  Write $(-1)^\mathcal{T}$ for the sign of
this permutation.  Now define
\begin{equation}
  \label{eq:PAC}
  \mathcal P_{A, B} \;=\; \sum  (-1)^\mathcal{T}
  \prod_{\tau\in\mathcal{T}} \sigma(\tau)
\end{equation}
with the summation being over all systems of treks $\mathcal{T}$ from
$A$ to $B$ with no sided intersection; recall the definition of trek
monomials from~\eqref{eq:trek-monomial}.  

\begin{theorem}[\cite{DST}]\label{TrekSystems} Suppose the underlying
  graph $G$ is acyclic.  Then the determinant of $\Sigma_{A, B}$
  equals $\mathcal P_{A, B}$.
\end{theorem}

This result admits a generalization to the case where the
graph $G$ contains directed cycles.  Indeed, \cite{DST} give a
rational expression for the determinant of $\Sigma_{A,B}$ in terms of
self-avoiding trek flows, which reduce to trek systems without sided
intersection in the acyclic case.  As this generalization is more
involved, we will not give any details here and focus instead on acyclic
graphs only.

We now extend the combinatorial characterization of determinants and
the trek separation result from Theorem~\ref{thm:trek_separation} to allow for {\em restricted} treks.
% , for which it
% is important to distinguish between the left and the right side of a
% trek.

\begin{definition}\label{dfn:restricted_treks} Let $A$, $B$, $P$, and $Q$
  be subsets of vertices of a mixed graph $G$.  A {\em
    $(P,Q)$-restricted trek} between $A$ and $B$ is a trek between a
  vertex in $A$ and a vertex in $B$ that has its left side in $P$ and
  its right side in $Q$.  Let $S_L$ and $S_R$ be two further subsets
  of vertices.  Then $A$ and $B$ are {\em $(P,Q)$-restricted trek-separated} by $(S_L, S_R)$ if every $(P, Q)$-restricted trek
  between $A$ and $B$ intersects $S_L$ on its left side or $S_R$ on its right side.
\end{definition}

\begin{example}\label{ex:verma1} Consider the Verma graph from
  Figure~\ref{fig:iv-verma-ommen-mixed}(b).  Take $A = \{2,4\}$,
  $B = \{2,3\}$, $P = \{2,4\}$, and $Q = \{2,3,4\}$.  Then $A$ and
  $B$ are $(P,Q)$-restricted trek-separated by $(\{\}, \{2\})$.
  Indeed, every trek between $A$ and $B$ that only uses $P$ on the
  left and only uses $Q$ on the right has to go through $2$ on the
  right. Note, however, that this is not true if, for example,
  $P = Q = V$ or if $3\in P$.
\end{example}

The main observation of this section is that restricted trek
separation is equivalent to a rank constraint on a special matrix.
Note also that part (ii) of the theorem is a direct generalization of 
Theorem~\ref{TrekSystems} to the restricted case.

\begin{theorem}\label{thm:main}
  Let $G=(V,\mathcal{D},\mathcal{B})$ be an acyclic mixed graph, and
  let $\Lambda\in\mathbb{R}^\mathcal{D}$ and
  $\Omega\in\mathit{PD}(\mathcal{B})$.  For $P,Q\subseteq V$, consider
  the matrix
  \[
    \Sigma^{(P,Q)} = \left[(I-\Lambda)_{P, P}\right]^{-T}
    \Omega_{P, Q}\left[(I-\Lambda)_{Q, Q}\right]^{-1},
  \]
  and its submatrix $\Sigma^{(P,Q)}_{A,B}$ for a choice of row indices
  $A\subseteq
  P$ and column indices $B\subseteq Q$.
  \begin{enumerate}
  \item[(i)] The rank of  $\Sigma^{(P,Q)}_{A,B}$ is at most
    $$ %\text{rank}(\Sigma^{(P,Q)}_{A,B})\leq
    % \min\{|S_L| + |S_R|:  (S_L, S_R) \; (P, Q)\text{-restricted
    %   trek-separate } A \text{ and } B\},
        \min\{|S_L| + |S_R|:   A
        \text{ and } B \text{ are } \; (P, Q)\text{-restricted trek-separated by } (S_L,S_R)\},
    $$
    and equal to this minimum generically.
  \item[(ii)] 
    If $|A|=|B|$, then the determinant of
    $\Sigma^{(P,Q)}_{A,B}$ is equal to
    \begin{equation*}
      \label{eq:PACPQ}
      \mathcal P_{A, B,(P,Q)} \;=\; \sum (-1)^\mathcal{T}
      \prod_{\tau\in\mathcal{T}} \sigma(\tau)
    \end{equation*}
    where the summation runs over all systems of treks $\mathcal{T}$
    that comprise only $(P,Q)$-restricted treks from $A$ to $B$ and have
    no sided intersection.
  \end{enumerate}
% %Then, if $A\subseteq P, B\subseteq Q,$ and $\#A = \#B$, $\det(N_{A, B})$ is the sum over all $(P, Q)$-restricted trek systems between $A$ and $B$ with no sided intersection.
% The the rank of its submatrix with rows $A\subseteq P$ and columns $B\subseteq Q$ satisfies
% $$\text{rank}(\Sigma^{(P,Q)}_{A,B})\leq \min\{\# S_L + \# S_R: (S_L, S_R)\, (P, Q)\text{-restricted } \text{trek separates } A \text{ from }B\},$$
% and equality holds for generic $\Sigma^{(P,Q)}$ consistent with the graph $G$.
\end{theorem}

The proof of Theorem~\ref{thm:main} is located in
Appendix~\ref{sec:restricted-treks-appendix}.  It is analogous to the
proofs of Theorems~\ref{thm:trek_separation} and~\ref{TrekSystems} as
developed in~\cite{STD,DST}.   

\begin{example}\label{ex:verma2} Consider once more the Verma graph
  from Figure~\ref{fig:iv-verma-ommen-mixed}(b).  Let $A = \{2,4\}$,
  $B = \{2,3\}$, $P = \{2,4\}$, and $Q = \{2,3,4\}$. We saw in
  Example~\ref{ex:verma1} that $A$ and $B$ are $(P, Q)$-restricted
  trek-separated 
  by $(\{\}, \{2\})$.  Now consider the matrix 
  \begin{align*}
    \Sigma^{(P,Q)} &\;=\;\left[ (I-\Lambda)_{24, 24}\right]^{-T} \Omega_{24,
    234} \left[ (I-\Lambda)_{234, 234}\right]^{-1}\\
    &\;=\;\begin{pmatrix}\omega_{22} & \omega_{22}\lambda_{23} & \omega_{22}\lambda_{23}\lambda_{34} + \omega_{24}\\
\omega_{24} & \omega_{24}\lambda_{23} & \omega_{24}\lambda_{23}\lambda_{34} + \omega_{44}
\end{pmatrix}.
  \end{align*}
  As predicted by Theorem~\ref{thm:main}(i), the submatrix
  $\Sigma^{(P,Q)}_{A,B} =\Sigma^{(P,Q)}_{24,23}$ has rank $1$.
\end{example}
\smallskip

%%% Local Variables:
%%% TeX-master: "nested_dets"
%%% End:

\section{Nested determinants}
\label{sec:nested-determinants}

In this section we demonstrate how restricted trek separation may lead to
polynomial equations in the vanishing ideal $\mathcal{I}(G)$ of the
model $\mathcal M(G)$ of an acyclic mixed graph $G$. These equations are in general not
determinantal, instead they are given by specific types of nested
determinants. In Section~\ref{sec:5.1} we introduce a \emph{swapping
  property} based on which in Theorem~\ref{thm:one_nested} we show how restricted trek separation
gives rise to the vanishing of such nested determinants.  In
Section~\ref{sec:5.2}, we show how the swapping property and Theorem~\ref{thm:one_nested} are sufficient to explain
the vanishing of the parentally nested determinants from Proposition~\ref{prop:ancestral-nested-det} and
Theorem~\ref{thm:cut-ancestral} in terms of restricted trek separation.

\subsection{Restricted trek separation and nested determinants}\label{sec:5.1}

We begin by defining a {\em swapping property} that allows us to factor certain subdeterminants of $\Sigma$. Recall that $\uplus$ denotes disjoint union (of multisets).
\begin{definition}\label{def:swapping} Let $A_1,\ldots, A_k,
  B_1,\ldots, B_k$ be sets of vertices of an acyclic mixed graph $G$
  with $|A_i| = |B_i|$ for every $i=1,\ldots, k$. Suppose {\em
    every} system of treks without sided intersection between
  $A_1\uplus\cdots\uplus A_k$ and $B_1\uplus\cdots\uplus B_k$ connects
  $A_i$ to $B_i$ for every $i$. Moreover, suppose that for any two
  trek systems $\mathcal T_1$ and $\mathcal T_2$ without sided
  intersection between $A_1\uplus\cdots\uplus A_k$ and
  $B_1\uplus\cdots\uplus B_k$ if we swap the treks between $A_i$ and
  $B_i$ from $\mathcal T_1$ with those from $\mathcal T_2$, then we
  obtain another two trek systems with no sided intersection between
  $A_1\uplus\cdots\uplus A_k$ and $B_1\uplus\cdots\uplus B_k$. Then,
  we say that $(A_1, B_1), \ldots,(A_k,B_k)$ satisfy the {\em swapping
    property}.
\end{definition}

The next lemma, which is proven in
Appendix~\ref{app:pf_lem_swapping}, shows how the swapping property
gives factorizations of subdeterminants of $\Sigma$ into different systems
of trek monomial sums.
\begin{lemma}\label{lem:swapping} Assume that $(A_1, B_1), \ldots,(A_k,B_k)$ satisfy the swapping property. Then,
$$|\Sigma_{A_1\uplus\cdots\uplus A_k, B_1\uplus\cdots\uplus B_k}| = \prod_{i=1}^k\mathcal P_{A_i, B_i, (C_i, D_i)},$$
where $C_1,\ldots, C_k, D_1,\ldots, D_k$ are sets of vertices determined
by the trek systems without sided intersection between
$A_1\uplus\cdots\uplus A_k$ and $B_1\uplus\cdots\uplus B_k$.
\end{lemma}

We now proceed to the main result of this section, Theorem~\ref{thm:one_nested}, which shows that
the swapping property for suitable sets of vertices implies that
certain nested determinants can be factored into sums of restricted
trek monomial systems. Later in the section we will see that though the conditions
of this theorem appear to be quite special, they are very natural. In particular,
they apply to a multitude of examples, and moreover, the theorem generalizes
our results from Proposition~\ref{prop:ancestral-nested-det} and Theorem~\ref{thm:cut-ancestral}.

\begin{theorem}\label{thm:one_nested} Let $a_1,\dots, a_n, b_1, \dots,
  b_n\in V$, and let $A_1,\dots, A_n$, $B_1,\dots, B_n$,
  $ C_1,\ldots, C_n$, $D_1,\ldots, D_n\subseteq V$ be such that
  $|A_i| = |B_i|$, $|C_j| = |D_j|$ for all $i, j$. Assume further that
  for each $i,j$, the sets $(A_i, B_i)$, $(C_j, D_j)$, and
  $(\{a_i\}, \{b_j\})$ satisfy the swapping property such that
  $$|\Sigma_{A_i\uplus C_j\uplus \{a_i\}, B_i\uplus D_j\uplus
    \{b_j\}}| = \mathcal P_{A_i, B_i, (P_i, Q_i)}\mathcal P_{C_j, D_j,
    (R_j, S_j)}\mathcal P_{a_i, b_j, (E_{ij}, F_{ij})}$$ for sets
  of vertices $P_i, Q_i, R_j, S_j, E_{ij}, F_{ij}$ and every $i$ and
  $j$. Assume also that
\begin{align}\label{eq:E_ijF_ijEF}
\left|\big(\mathcal P_{a_i, b_j, (E_{ij}, F_{ij})}\big)_{1\le i,j\le n}\right| = \mathcal P_{\{a_1,\ldots, a_n\}, \{b_1,\dots, b_n\}, (E, F)}
\end{align}
for some $E, F$. Then,
\begin{align*}
&\left|\big(|\Sigma_{A_i\uplus C_j\uplus \{a_i\}, B_i\uplus D_j\uplus \{b_j\}}|\big)_{1\leq i,j\leq n}\right|=\quad\quad\quad\quad\quad\quad\quad\quad\quad\quad\quad\quad\quad\quad \\
&= \left(\prod_i\mathcal P_{A_i,B_i, (P_i, Q_i)}\right) \left(\prod_j\mathcal P_{C_j, D_j, (R_j, S_j)}\right)\mathcal P_{\{a_1,\dots, a_n\}, \{b_1,\dots, b_n\}, (E, F)}.
\end{align*}
% the sum of the product of trek monomials of all trek systems between $\{a_1,\dots, a_n\}$ and $\{b_1,\dots, b_n\}$ with no sided intersection 
%%such that any trek between $a_i$ and $b_j$ appearing in such a system can be part of a trek system between $A_i\cup a_i$ and $B_i\cup b_j$, of the corresponding trek monomials, 
%that only use vertices from $E$ on the left and from $F$ on the right (i.e. $\mathcal P_{\{a_1,\dots, a_n\}, \{b_1,\dots, b_n\}, (E, F)})$,
%multiplied by the trek monomials $\prod_i\mathcal P_{A_i,B_i, (C_i, D_i)}$. %[need some more conditions for this to be true].
\end{theorem}

\begin{proof} Let $M$ be the matrix with $i,j$-th entry equal to
  $|\Sigma_{A_i\uplus C_j\uplus \{a_i\}, B_i\uplus D_j\uplus
    \{b_j\}}|$ for $1\leq i,j\leq n$. Since the entries in the $i$-th
  row of $M$ are divisible by $\mathcal P_{A_i, B_i, (P_i, Q_i)}$, we
  can factor the determinant of $M$ as
$$|M| = \left(\prod_{i=1}^n \mathcal P_{A_i, B_i, (P_i, Q_i)} \right) \det ((\mathcal P_{C_j, D_j, (R_j, S_j)}\mathcal P_{a_i, b_j, (E_{ij}, F_{ij})})_{1\leq i,j\leq n}).$$
Since the $j$-th column is divisible by $\mathcal P_{C_j, D_j, (R_j,
  S_j)}$, we can further factor as
\begin{align}
  \nonumber
  |M| &= \left(\prod_i\mathcal P_{A_i,B_i, (P_i, Q_i)}\right) \left(\prod_j\mathcal P_{C_j, D_j, (R_j, S_j)}\right)\det(\mathcal P_{a_i, b_j, (E_{ij}, F_{ij})})\\
  \label{eq:factorization}
      &=\left(\prod_i\mathcal P_{A_i,B_i, (P_i, Q_i)}\right)
    \left(\prod_j\mathcal P_{C_j, D_j, (R_j, S_j)}\right)\mathcal
    P_{\{a_1,\dots, a_n\}, \{b_1,\dots, b_n\}, (E, F)}
\end{align}
as required.
\end{proof}

Condition~\eqref{eq:E_ijF_ijEF} deserves further
discussion. % In Lemma~\ref{lem:E_ijF_ij}, we show that it holds for many choices of sets $E_{ij}$ and $F_{ij}$.
By Theorem~\ref{thm:main}(ii),
\[
  \left|\big(\mathcal P_{a_i, b_j, (E, F)}\big)_{i,j}\right| = \mathcal P_{\{a_1,\ldots,
    a_n\}, \{b_1,\dots, b_n\}, (E, F)}
\] for every $a_1,\dots, a_n, b_1, \dots, b_n\in V$ and
$E, F\subseteq V$. However, for this equality, it is not necessary to
have the exact sets $E$ and $F$ for all matrix entries on the
left-hand side.  Instead, slightly different sets $E_{ij}$ and
$F_{ij}$ can be used as the following lemma indicates.

\begin{lemma}\label{lem:E_ijF_ij} Let $a_1,\dots, a_n, b_1, \dots, b_n\in V$ and $E, F\subseteq V$. Then,
\begin{align}\label{eq:E_ijF_ij}
\left|\big(\mathcal P_{a_i, b_j, (E_{ij}, F_{ij})}\big)_{i,j}\right| = \mathcal P_{\{a_1,\ldots, a_n\}, \{b_1,\dots, b_n\}, (E, F)}
\end{align}
if
$$E\setminus\{a_k: k\neq i\} \subseteq E_{ij}\subseteq E\cup(V\setminus\an(a_i))\,\,\text{ and }\,\, F\setminus\{b_\ell: \ell\neq j\}\subseteq F_{ij}\subseteq F\cup(V\setminus\an(b_j)).$$
\end{lemma}
The proof can be found in Appendix~\ref{app:pf_lem_E_ijF_ij}.
We remark that additional choices of $E_{ij}$ and $F_{ij}$ are certainly possible depending on the graph structure at hand. Nevertheless, this lemma gives us a wide variety of sets $E_{ij}$ and $F_{ij}$ for which equality~\eqref{eq:E_ijF_ij} holds.

\begin{example}\label{ex:Verma5.5} Recall the Verma graph from Figure
\ref{fig:iv-verma-ommen-mixed}(b). In the context of Theorem \ref{thm:one_nested}
let
\begin{align*}
  a_1 &= 2, \quad a_2 = 3, \quad b_1 = 2, \quad b_2 = 4, \\
  C_1 &= \{1\}, \quad  C_2 = \{1\}, \quad D_1 = \{1\}, \quad D_2 = \{3\}, \quad A_1=A_2=B_1=B_2=\emptyset.
\end{align*}
Then $(C_1,D_1), (\{a_i\},\{b_1\})$ satisfy the swapping property for
$i=1,2$.  The same is true for $(C_2,D_2), (\{a_i\}, \{b_2\})$ for $i=1,2$. Moreover,
$$ |\Sigma_{12, 12}|=  \mathcal{P}_{1,1}\mathcal{P}_{2,2,(\{234\},\{234\})},\quad 
|\Sigma_{12,34}| =  \mathcal{P}_{1,3}\mathcal{P}_{2,4,(\{234\},\{24\})},$$
$$|\Sigma_{13, 12}|= \mathcal{P}_{1,1}\mathcal{P}_{3,2,(\{234\},\{234\})}, \quad
|\Sigma_{13,34}|= \mathcal{P}_{1,3}\mathcal{P}_{3,4,(\{234\},\{24\})}.$$
By Lemma~\ref{lem:E_ijF_ij}, we have that
$$\begin{vmatrix}
                   \mathcal{P}_{2,2,(\{234\},\{234\})} & \mathcal{P}_{2,4,(\{234\},\{24\})} \\
                    \mathcal{P}_{3,2,(\{234\},\{234\})} &
                    \mathcal{P}_{3,4,(\{234\},\{24\})}
                  \end{vmatrix} =\mathcal{P}_{\{2,3\},\{2,4\},(\{234\},\{24\})}.$$
Therefore, the conditions of Theorem \ref{thm:one_nested} are satisfied, and
\begin{align*}
  %\label{VermaMatrix1-again}
  f_{\text{Verma}} = \begin{vmatrix}
    |\Sigma_{12, 12}| & |\Sigma_{12,34}|\\
    |\Sigma_{13, 12}| & |\Sigma_{13,34}|
  \end{vmatrix} &= \mathcal{P}_{1,1}\mathcal{P}_{1,3}\mathcal{P}_{\{2,3\},\{2,4\},(\{234\},\{24\})}.
\end{align*}
Now $\mathcal{P}_{\{2,3\},\{2,4\},(\{2,3,4\},\{2,4\})}$ is zero because
$\{2,3\}$ and $\{2,4\}$ are $(\{234\}, \{2,4\})$-restricted trek
separated by $(\{2\}, \emptyset)$. That is, the nested determinant
giving $f_{\text{Verma}}$ vanishes because treks between $\{2, 3\}$ and $\{2,4\}$
that only use $\{2,3,4\}$ on the left and $\{2,4\}$ on the right must all
pass through $2$ on the left.
%Indeed, if we consider the matrix
%\begin{align*}
%  \left( (I-\Lambda)_{24, 24}\right)^{-T} \Omega_{24,
%  234} \left( (I-\Lambda)_{234, 234}\right)^{-T} =
%  \begin{pmatrix}
%    \omega_{22} & \omega_{22}\lambda_{23} &
%    \omega_{22}\lambda_{23}\lambda_{34} + \omega_{24}\\ \omega_{24} &
%    \omega_{24}\lambda_{23} & \omega_{24}\lambda_{23}\lambda_{34} +
%    \omega_{44}
%\end{pmatrix},
%\end{align*}
%one may check that its first $2\times 2$ minor equals
%$\mathcal{P}_{\{2,3\},\{2,4\},(\{1\},\{1,3\})}$.

\end{example}

The following corollary gives a combinatorial interpretation of the
vanishing of a nested determinant like the ones specified in
Theorem~\ref{thm:one_nested}.

\begin{corollary}\label{cor:restricted_treks}
  Suppose the conditions in Theorem~\ref{thm:one_nested} are
  satisfied.  Define matrix entries
  $M_{ij} =|\Sigma_{A_i\uplus C_j\uplus \{a_i\}, B_i\uplus D_j\uplus
    \{b_j\}}|$. Then, $|M| = 0$ if and only if at least one of the
  following holds:
\begin{itemize}
\item[(i)] The sets $\{a_1,\ldots, a_n\}$ and $\{b_1,\ldots, b_n\}$ are $(E,F)$-restricted trek separated by some sets $(X,Y)$ with $|X| + |Y| < n$; or
\item[(ii)] For some $i=1,\ldots, n$, the sets $A_i$ and $B_i$ are $(P_i, Q_i)$-restricted trek separated by some sets $(X,Y)$ with $|X| + |Y| < |A_i|$; or
\item[(iii)] For some $j=1,\ldots, n$, the sets $C_j$ and $D_j$ are $(R_j, S_j)$-restricted trek separated by some sets $(X,Y)$ with $|X| + |Y| < |C_j|$.
\end{itemize}
\end{corollary}

\begin{proof} By Theorem~\ref{thm:one_nested}, $|M|$
  factors as in~\eqref{eq:factorization}.  Therefore, $|M| = 0$ if and
  only if one of the factors in this expression vanishes. By
  Theorem~\ref{thm:main},  each of these factors vanishes
  if and only if the corresponding restricted trek separation is
  satisfied.
\end{proof}

%\subsubsection{The Verma constraint 1.}
%
%Now by Corollary~\ref{cor:restricted_treks} $\mathcal{P}_{\{2,3\},\{2,4\},(\{1\},\{1,3\})} = 0$ as $\{2,3\}$ and $\{2,4\}$ are restricted-$(\{1\}, \{1,3\})$ by $(\{2\},\emptyset)$. 

%\subsubsection{The Verma constraint 2.}
%
%Again consider the Verma graph from Figure
%\ref{fig:iv-verma-ommen-mixed}(b). Letting
%\begin{align*}
%  a_1 &= 1, \quad a_2 = 3, \quad b_1 = 1, \quad b_2 = 2, \\
%  A_2 &= \{2,3\}, \quad  B_2 = \{3,4\}, \quad A_1=B_1=C_1=C_2=D_1=D_2=\emptyset,
%\end{align*}
%one may check the conditions of Theorem \ref{thm:one_nested} to find that
%\begin{align*}
%  %\label{VermaMatrix2-again}
%  f_{\text{Verma}} = \begin{vmatrix}
%    \sigma_{11} & \sigma_{12}\\ 
%    |\Sigma_{231,341}| & |\Sigma_{231,342}|
%  \end{vmatrix}
%                         &=
%                         \begin{vmatrix}
%                           \mathcal{P}_{1,1} & \mathcal{P}_{1,2}\\ 
%                           \mathcal{P}_{1,1}\mathcal{P}_{\{2,3\},\{3,4\},(\{1\},\{2\})} &
%                           \mathcal{P}_{1,2}\mathcal{P}_{\{2,3\},\{3,4\},(\{1\},\{2\})}
%                         \end{vmatrix} \\
%  &= \mathcal{P}_{\{2,3\},\{3,4\},(\{1\},\{2\})}\mathcal{P}_{\{1,1\},\{1,2\}}.
%\end{align*}
%This shows that once we ``account for'' the treks between $\{2,3\}$
%and $\{3,4\}$ which avoid $(\{1\},\{2\})$ the Verma constraint arises
%simply from a repeated row index giving
%$\mathcal{P}_{\{1,1\},\{1,2\}}=0$.
%

\subsection{Restricted trek separation and ancestral
  vertices}\label{sec:5.2}

We return to the parentally nested determinants and specifically the
nested determinant $f_{ij}$ defined in~\eqref{eq:nested-det-parent}. In
Proposition~\ref{prop:ancestral-nested-det} we gave conditions that
entailed the vanishing of $f_{ij}$.  We now see how this result is also implied by
restricted trek separation.

\begin{proposition}\label{prop:restricted_trek_ancestral}  Consider
  the conditions from Proposition~\ref{prop:ancestral-nested-det},
  i.e., $i$ is a vertex in $G=(V, \mathcal{D}, \mathcal{B})$
  satisfying 
   \begin{enumerate}%[label=(\roman*)]
  \item[(i)] $\pa(i)\cap\sib(i)=\emptyset$,
  \item[(ii)] all vertices in $\pa(i)$ are ancestral, and
  \item[(iii)] the set $J$ of all ancestral vertices in $V\setminus(\pa(i)\cup\sib(i)\cup\{i\})$ is non-empty.
  \end{enumerate}
  Then for every $j\in J$ the sets $\pa(i)\cup\{j\}$ and $\pa(i)\cup\{i\}$ are
  $(\pa(i)\cup\{j\},V)$-restricted trek separated by
  $(\emptyset, \pa(v))$.  This restricted trek separation implies that
  $f_{ij}\in\mathcal{I}(G)$ for all $j\in J$, i.e., all parentally nested determinants for $(i,J)$ lie in $\mathcal I(G)$.  
\end{proposition}
The proof can be found in Appendix~\ref{app:pf_prop_restricted_trek_ancestral}.

\begin{example} %[The Verma constraint revisited]
Again consider the Verma graph from Figure
\ref{fig:iv-verma-ommen-mixed}(b). We have that 
%Letting
%\begin{align*}
%  a_1 &= 1, \quad a_2 = 3, \quad b_1 = 1, \quad b_2 = 2, \\
%  A_2 &= \{2,3\}, \quad  B_2 = \{3,4\}, \quad A_1=B_1=C_1=C_2=D_1=D_2=\emptyset,
%\end{align*}
%one may check the conditions of Theorem \ref{thm:one_nested} to find that
\begin{align*}
  %\label{VermaMatrix2-again}
  f_{\text{Verma}} = \begin{vmatrix}
    \sigma_{13} & \sigma_{14}\\ 
    |\Sigma_{123,123}| & |\Sigma_{123,124}|
  \end{vmatrix}.
%                         &=
%                         \begin{vmatrix}
%                           \mathcal{P}_{1,1} & \mathcal{P}_{1,2}\\ 
%                           \mathcal{P}_{1,1}\mathcal{P}_{\{2,3\},\{3,4\},(\{1\},\{2\})} &
%                           \mathcal{P}_{1,2}\mathcal{P}_{\{2,3\},\{3,4\},(\{1\},\{2\})}
%                         \end{vmatrix} \\
%  &= \mathcal{P}_{\{2,3\},\{3,4\},(\{1\},\{2\})}\mathcal{P}_{\{1,1\},\{1,2\}}.
\end{align*}
As mentioned in Section~\ref{sec:ancestral}, $f_{\text{Verma}} = f_{41}$. Indeed, $(\{3,1\}, \{3,4\})$ are $(\{3,1\}, V)$-restricted trek separated by $(\emptyset, \{3\})$. 
%This shows that once we ``account for'' the treks between $\{2,3\}$
%and $\{3,4\}$ which avoid $(\{1\},\{2\})$ the Verma constraint arises
%simply from a repeated row index giving
%$\mathcal{P}_{\{1,1\},\{1,2\}}=0$.
\end{example}

\begin{example}%[Ancestral graph from Richardson and Spirtes 2002]
\begin{figure}[t]
  \centering
  \tikzset{
    every node/.style={circle, draw,inner sep=1mm, minimum size=0.55cm, draw, thick, black, fill=gray!20, text=black},
    every path/.style={thick}
  }
  \scalebox{0.85}{
    (a) \;
    \begin{tikzpicture}[align=center,node distance=2.2cm]
      \node [] (1) {1};
      \node [] (2) [right of=1] {2};
      \node [] (3) [below of=1]    {3};
      \node [] (4) [below of=2]    {4};

      \draw[blue] [-latex] (1) edge (3);
      \draw[blue] [-latex] (2) edge (4);
      % \draw[red] [latex-latex, bend left] (2) edge (4);
      \draw[red] [latex-latex] (1) edge (4);
      \draw[red] [latex-latex] (1) edge (2);
      \draw[red] [latex-latex] (2) edge (3);
    \end{tikzpicture}
    \qquad 
    (b) \;
    \begin{tikzpicture}[align=center,node distance=2.2cm]
      \node [] (1) {1};
      \node [] (3) [below right of=1] {3};
      \node [] (2) [above right of=3] {2};
      \node [] (4) [right of=2] {4};
      \node [] (6) [below right of=4] {6};
      \node [] (5) [above right of=6] {5};

      \draw[blue] [-latex] (1) edge (3);
      \draw[blue] [-latex] (2) edge (3);
      \draw[blue] [-latex] (4) edge (6);
      \draw[blue] [-latex] (5) edge (6);
      \draw[red] [latex-latex, bend left] (1) edge (4);
      \draw[red] [latex-latex, bend left] (2) edge (5);
      \draw[red] [latex-latex, bend left] (1) edge (5);
      \draw[red] [latex-latex] (4) edge (5);
      \draw[red] [latex-latex] (1) edge (2);
      \draw[red] [latex-latex] (2) edge (4);
      \draw[red] [latex-latex] (1) edge (6);
      \draw[red] [latex-latex] (2) edge (6);
      \draw[red] [latex-latex] (3) edge (5);
      \draw[red] [latex-latex] (3) edge (4);
    \end{tikzpicture}
  }
  \vspace{.3cm}
  
  \scalebox{0.85}{
    (c) \;
    \begin{tikzpicture}[align=center,node distance=2.2cm]
      \node [] (1) {1};
      \node [] (3) [below right of=1] {3};
      \node [] (2) [above right of=3] {2};
      \node [] (5) [below right of=2] {5};
      \node [] (4) [above right of=5] {4};

      \draw[blue] [-latex] (1) edge (3);
      \draw[blue] [-latex] (2) edge (3);
      \draw[blue] [-latex] (2) edge (5);
      \draw[blue] [-latex] (4) edge (5);
      \draw[red] [latex-latex, bend left] (1) edge (4);
      \draw[red] [latex-latex] (1) edge (5);
      \draw[red] [latex-latex] (1) edge (2);
      \draw[red] [latex-latex] (2) edge (4);
      \draw[red] [latex-latex] (3) edge (4);
    \end{tikzpicture}
  }
  \caption{Ancestral graph examples.}\label{fig:ancestral-graphs}
\end{figure}
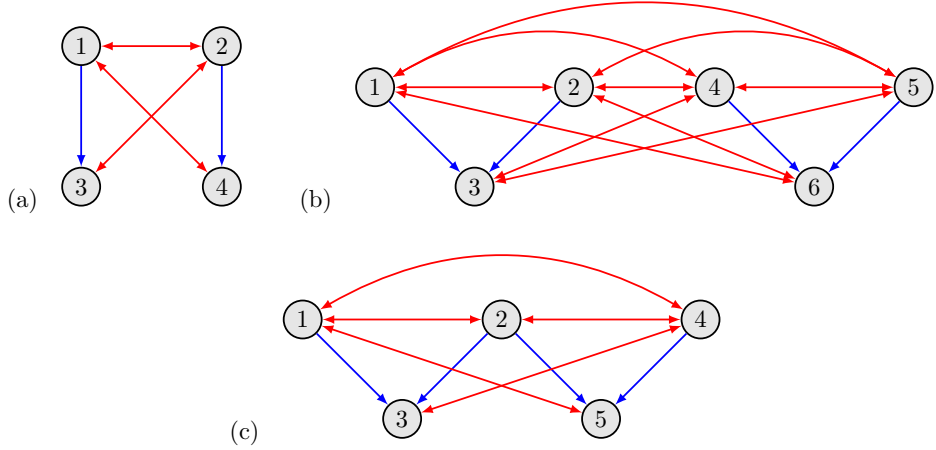

Consider the ancestral graph from Figure \ref{fig:ancestral-graphs}(a), which was studied in more detail in~\cite{richardson:2002}. Applying Proposition~\ref{prop:restricted_trek_ancestral}, we choose $i=3$ and $j=4$ to obtain that the corresponding polynomial $f_{34}$ vanishes. Indeed, by Lemma~\ref{lem:swapping}, we have that
\begin{align*}
 f_{34} &= \begin{vmatrix}
    \sigma_{11} & \sigma_{13} \\ 
    |\Sigma_{24,21}| & |\Sigma_{24,23}|
  \end{vmatrix}\\
                      & =
                       \begin{vmatrix}
                         \mathcal{P}_{1,1,(\{1,3,4\},\{1,2,3,4\})} & \mathcal{P}_{1,3,(\{1,3,4\},\{1,2,3,4\})} \\ 
                         \mathcal{P}_{2,2}\mathcal{P}_{4,1,(\{1,3,4\},\{1,2,3,4\})} &
                         \mathcal{P}_{2,2}\mathcal{P}_{4,3,(\{1,3,4\},\{1,2,3,4\})}
                       \end{vmatrix}                   \;=\;\mathcal{P}_{2,2}^2\mathcal{P}_{\{1,4\},\{1,3\},(\{1,3,4\},\{1,2,3,4\})}. 
\end{align*}
As all treks between $\{1,4\}$ and $\{1,3\}$ which avoid $2$ on the left must intersect $1$ on the right we have that $\mathcal{P}_{\{1,4\},\{1,3\},(\{1,3,4\},\{1,2,3,4\})}=0$ and thus the above nested determinant is an element of the vanishing ideal for the graph. It can be checked by computational algebra that the above determinant generates the ideal of the model.
\end{example}

%Here, $\Sigma_{A,B}$ denotes the determinant of the $A\times B$
%submatrix of the symmetric matrix of indeterminates
%$\Sigma=(\sigma_{ij})$.  We also use the usual shorthands when
%referring to indices, so $1\equiv \{1\}$, $12=\{1,2\}$, etc.

\begin{example} %[Ancestral graph with larger parent sets]
For the ancestral graph in Figure \ref{fig:ancestral-graphs}(b), we choose $i=6, j=3$. Then, using Lemma~\ref{lem:swapping}, we have that
\begin{align*}
 f_{63}= \begin{vmatrix}
    |\Sigma_{123,124}| & |\Sigma_{123,125}| & |\Sigma_{123,126}| \\
    \sigma_{44} & \sigma_{45} & \sigma_{46}\\ 
    \sigma_{54} & \sigma_{55} & \sigma_{56}
  \end{vmatrix}
                                          &= \mathcal{P}_{\{1,2\},\{1,2\}} \mathcal{P}_{\{3,4,5\},\{4,5,6\},(\{3,4,5,6\},\{3,4,5,6\})}.
\end{align*}
As all treks between $\{3,4,5\}$ and $\{4,5,6\}$ which are
restricted to only use $\{3,4,5,6\}$ on their left or right sides must use
4 or 5 on their right side it follows that
$$\mathcal{P}_{\{3,4,5\},\{4,5,6\},(\{3,4,5,6\},\{3,4,5,6\})} = 0$$ and thus the
above determinant is an element of the vanishing ideal.
\end{example}
\begin{example} %[Ancestral graph with intersecting parent sets]
Consider the graph from Figure \ref{fig:ancestral-graphs}(c). Lemma~\ref{lem:swapping} implies that
\begin{align*}
  f_{53}=\begin{vmatrix}
    |\Sigma_{123,122}| & |\Sigma_{123,124}| & |\Sigma_{123,125}|\\
    \sigma_{22} & \sigma_{24} & \sigma_{25}\\ 
    \sigma_{42} & \sigma_{44} & \sigma_{45}
  \end{vmatrix}
  = \mathcal{P}_{\{1,2\},\{1,2\}} \mathcal{P}_{\{3,2,4\},\{2,4,5\},(\{2,3,4,5\},\{1,2,3,4,5\})}.
  %\;=\;
  % \begin{vmatrix}
  %   0 & \Sigma_{123,124} & \Sigma_{123,125}\\
  %   \Sigma_{2,2} & \Sigma_{2,4} & \Sigma_{2,5}\\ 
  %   \Sigma_{4,2} & \Sigma_{4,4} & \Sigma_{4,5}
  % \end{vmatrix}
  % \;=\;\\
  % -\Sigma_{123,124}\Sigma_{24,25}+\Sigma_{123,125}\Sigma_{24,24}
  % \;=\;
  % \begin{vmatrix}
  %   \Sigma_{24,24} & \Sigma_{24,25}\\ 
  %   \Sigma_{123,124} & \Sigma_{123,125}
  % \end{vmatrix}.
\end{align*}
As any trek from $\{3,2,4\}$ to $\{2,4,5\}$ which avoids $1$ on
the left must use $2$ or $4$ on the right it follows
that $\mathcal{P}_{\{3,2,4\},\{2,4,5\},(\{2,3,4,5\},\{1,2,3,4,5\})} = 0$. Hence
$f_{53}$ is in (and generates) the vanishing ideal.
\end{example}

%%% Local Variables:
%%% TeX-master: "nested_dets"
%%% End:

%\input{cyclic}

\section{Beyond swapping:  Recursive nesting,
  directed cycles, and the pentad}
\label{sec:beyond-theorem}

%\subsection{Four node graph from Figure \ref{fig:ommen}(a)}

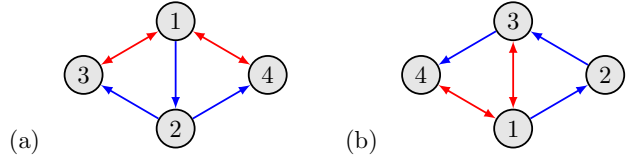
\begin{figure}[t]
  \centering
  \tikzset{
    every node/.style={circle, draw,inner sep=1mm, minimum size=0.55cm, draw, thick, black, fill=gray!20, text=black},
    every path/.style={thick}
  }
  \scalebox{0.85}{
    (a) \;
    \begin{tikzpicture}[align=center,node distance=2.2cm]
      \node [] (1) {1};
      \node [] (3) [below left=0.4cm and 1cm of 1] {3};
      \node [] (4) [below right=0.4cm and 1cm of 1] {4};
      \node [] (2) [below right=0.4cm and 1cm of 3] {2};

      \draw[blue] [-latex] (1) edge (2);
      \draw[blue] [-latex] (2) edge (3);
      \draw[blue] [-latex] (2) edge (4);
      \draw[red] [latex-latex] (1) edge (3);
      \draw[red] [latex-latex] (1) edge (4);
    \end{tikzpicture}
    \qquad 
    (b) \;
    \begin{tikzpicture}[align=center,node distance=2.2cm]
      \node [] (3) {3};
      \node [] (4) [below left=0.4cm and 1cm of 3] {4};
      \node [] (2) [below right=0.4cm and 1cm of 3] {2};
      \node [] (1) [below right=0.4cm and 1cm of 4] {1};

      \draw[blue] [-latex] (1) edge (2);
      \draw[blue] [-latex] (2) edge (3);
      \draw[blue] [-latex] (3) edge (4);
      \draw[red] [latex-latex] (1) edge (4);
      \draw[red] [latex-latex] (1) edge (3);
    \end{tikzpicture}
  }
  \caption{Two four-node graphs whose vanishing ideals are known from
    computational algebra. We are able to write the generators of
    these vanishing ideals as nested determinant but our combinatorial
    conditions do not appear to apply.}\label{fig:ommen}
\end{figure}
In this section we explore examples that are not covered by
Theorem~\ref{thm:one_nested} and Corollary~\ref{cor:restricted_treks}
but whose constraints are still nested determinants.
In Section~\ref{sec:5.3} we consider two mixed graphs for which the constraints could be
presented as nested determinants but for which---we argue---a
recursive nesting of determinants is more natural and more directly
tied to restricted trek separation. In Section~\ref{sec:6.2} we discuss an example of a
directed graph with a directed cycle for which restricted trek separation also implies a nested
determinant constraint. Finally, in Section~\ref{sec:6.3}, we turn to the pentad from factor analysis~\cite{MR2299716},
and show that it is also defined by a nested determinant.

\subsection{Restricted trek separation and determinants of recursively
  nested matrices}\label{sec:5.3}

% Some models are defined by the vanishing of recursively nested determinants. We will see in this section that such constraints can also be explained by restricted trek separation.

%Some of these cases can still be explained via restricted trek separation, but are given by determinants of doubly, triply, or more times nested matrices. 

It is apparent from Theorem~\ref{thm:one_nested} and Corollary~\ref{cor:restricted_treks} that (singly) nested determinants give a way to express restricted trek systems as factors as follows. Recall the Verma graph and Example~\ref{ex:Verma5.5} where we saw that 
\begin{align*}
  %\label{VermaMatrix1-again}
  f_{\text{Verma}} = \begin{vmatrix}
    |\Sigma_{12, 12}| & |\Sigma_{12,34}|\\
    |\Sigma_{13, 12}| & |\Sigma_{13,34}|
  \end{vmatrix} &= \begin{vmatrix} \mathcal{P}_{1,1}\mathcal{P}_{2,2,(\{2,3,4\},\{2,3,4\})} & 
\mathcal{P}_{1,3}\mathcal{P}_{2,4,(\{2,3,4\},\{2,4\})}\\
\mathcal{P}_{1,1}\mathcal{P}_{3,2,(\{2,3,4\},\{2,3,4\})} &
\mathcal{P}_{1,3}\mathcal{P}_{3,4,(\{2,3,4\},\{2,4\})}\end{vmatrix}\\
    &= \mathcal{P}_{1,1}\mathcal{P}_{1,3}\mathcal{P}_{\{2,3\},\{2,4\},(\{2,3,4\},\{2,4\})}.
\end{align*}
Notice that each of the entries in the $2\times 2$ matrix whose determinant equals $f_{\text{Verma}}$ is a determinant, which, because of the swapping property, factors out $\mathcal P_{1,1}$ or $\mathcal P_{1,3}$, and leaves monomials corresponding to restricted trek systems that only use $\{2,3,4\}$ on the left and $\{2,4\}$ on the right.

In other graphs, using single subdeterminants of $\Sigma$ is not enough to factor out restricted trek systems. However, one can instead use recursively nested determinants. We illustrate this in the following two examples.
\begin{example}
As is shown by van Ommen and Mooij, see \cite{VanOmmenMooij_UAI_17}
Appendix B, the graph from Figure \ref{fig:ommen}(a) has vanishing
ideal generated by
\begin{align*}
  f = p_0^2\sigma_{34} + p_0\sigma_{23}p_2 + p_1\sigma_{24}p_0 + p_1\sigma_{22}p_2
\end{align*}
where
\begin{align*}
  p_0 = |\Sigma_{12,12}|, \quad
  p_1 = |\Sigma_{13,21}|, \quad
  p_2 = |\Sigma_{21,14}|.
\end{align*}
One may express $f$ as a nested determinant by noting that
\begin{align*}
  -f =
  \begin{vmatrix}
    0 & |\Sigma_{12,12}| & |\Sigma_{12,14}|\\
    |\Sigma_{12,12}| & \sigma_{22} & \sigma_{24}\\
    |\Sigma_{13,12}| & \sigma_{32} & \sigma_{34} 
  \end{vmatrix}
                     &=
                       \begin{vmatrix}
                         |\Sigma_{112,112}| & |\Sigma_{12,12}| & |\Sigma_{12,14}|\\
                         |\Sigma_{12,12}| & \sigma_{22} & \sigma_{24}\\
                         |\Sigma_{13,12}| & \sigma_{32} & \sigma_{34}
                       \end{vmatrix}.
\end{align*}
While the above representation of $f$ suggests applying Theorem \ref{thm:one_nested} with
\begin{align*}
  a_1 &= 2, \quad a_2 = 2, \quad a_3 = 3, \quad b_1 = 2, \quad b_2 = 2, \quad b_3 = 4, \\
  A_1 &= \{1\}, \quad B_1 = \{1\}, \quad C_1 = \{1\}, \quad D_1 = \{1\}, \\
  A_i &= B_i = C_i = D_i = \emptyset \quad (i=2,3),
\end{align*}
this, unfortunately, does not seem to satisfy the conditions of the
theorem. On the other hand, we can express $f$ as the determinant of a
matrix whose entries are themselves nested determinants, namely,
$$f = \begin{vmatrix} |\Sigma_{12,12}| & |\Sigma_{12, 14}|\\
\begin{vmatrix} |\Sigma_{12, 12}| & |\Sigma_{12, 13}|\\
\Sigma_{22} &\Sigma_{23}
\end{vmatrix} &
\begin{vmatrix}|\Sigma_{12, 12}| & |\Sigma_{12,13}|\\
\Sigma_{42}&\Sigma_{43}
\end{vmatrix}
\end{vmatrix}.$$
Moreover, using Lemma~\ref{lem:swapping}, we get the following factorizations:
\begin{align*}
  |\Sigma_{12,12}|
  &= \mathcal P_{1,1}\mathcal P_{2,2, (\{2,3,4\}, \{2,3,4\})}, &\,
                                                                 |\Sigma_{12,
                                                                 13}|
  &= \mathcal P_{1,1}\mathcal P_{2,3,(\{2,3,4\},\{2,3,4\})},\,\\
  |\Sigma_{12,14}| &= \mathcal P_{1,1}\mathcal P_{2,4,(\{2,3,4\},
                     \{2,3,4\})}.
\end{align*}
 Thus,
 \begin{align*}
   \begin{vmatrix} |\Sigma_{12, 12}| & |\Sigma_{12, 13}|\\
\Sigma_{22} &\Sigma_{23}
\end{vmatrix}
              &= \mathcal P_{1,1}\omega_{2,2}\lambda_{12}\omega_{1,3} = \mathcal P_{1,1}\mathcal P_{2,2,(\{2,3,4\},\{2,3,4\})}\mathcal P_{3,2,(\{3,4\}, \{12,3,4\})},
   \\
   \begin{vmatrix}|\Sigma_{12, 12}| & |\Sigma_{12,13}|\\
\Sigma_{42}&\Sigma_{43}
\end{vmatrix}
             &= \mathcal
               P_{1,1}\omega_{2,2}\lambda_{1,2}\omega_{1,3}\lambda_{2,4}
               = \mathcal P_{1,1}\mathcal P_{2,2, (\{2,3,4\},
               \{2,3,4\})}\mathcal P_{3,4, (\{3,4\}, \{1,2,3,4\})}.
 \end{align*}
 This implies that
\begin{align*}
  f
  &= \mathcal P_{1,1}^2\mathcal P_{2,2, (\{2,3,4\}, \{2,3,4\})}\begin{vmatrix} P_{2,2, (\{2,3,4\}, \{2,3,4\})} & \mathcal P_{2,4, (\{2,3,4\}, \{2,3,4\})}\\
\mathcal P_{3,2, (\{34\}, \{1,2,3,4\})} &\mathcal P_{3,4, (\{34\}, \{1,2,3,4\})}
\end{vmatrix}\\
&= \mathcal P_{1,1}^2\mathcal P_{2,2, (\{2,3,4\}, \{2,3,4\})}\mathcal
                  P_{\{2,3\}, \{2,4\}, (\{2,3,4\}, \{1,2,3,4\})},
\end{align*}
where we have used Lemma~\ref{lem:E_ijF_ij} in the second step. Indeed, the
sets $\{2,3\}$ and $\{2,4\}$ are $(\{2,3,4\}, \{1,2,3,4\})$-restricted trek-separated by $(\emptyset, \{2\})$.
\end{example}

\begin{example}
Now consider the graph in Figure \ref{fig:ommen}(b). Its
vanishing ideal is generated by the polynomial
\begin{align*} 
  f=\begin{vmatrix}
    |\Sigma_{112,112}| & |\Sigma_{12,13}| & |\Sigma_{12,12}| \\
    |\Sigma_{13,12}| & \sigma_{33} & \sigma_{32} \\
    |\Sigma_{14,12}| & \sigma_{43} & \sigma_{42}
  \end{vmatrix}.
\end{align*}
While the above representation of $f$ suggests applying Theorem \ref{thm:one_nested},
%with
%\begin{align*}
%  a_1 &= 2, \quad a_2 = 3, \quad a_3 = 4, \quad b_1 = 2, \quad b_2 = 3, \quad b_3 = 2, \\
%  A_1 &= \{1\}, \quad B_1 = \{1\}, \quad C_1 = \{1\}, \quad D_1 = \{1\}, \\
%  A_i &= B_i = C_i = D_i = \emptyset \quad (i\in\{2,3\}),
%\end{align*}
this, unfortunately, does not seem to satisfy the conditions of the
theorem either. On the other hand, we can express $f$ as the
determinant of a matrix whose entries are themselves nested
determinants:
\[
  f=\begin{vmatrix}
    |\Sigma_{12,13}| & |\Sigma_{12,14}|\\ 
    \begin{vmatrix}
      \Sigma_{2,3} & \Sigma_{3,3}\\
      |\Sigma_{12,12}| &      |\Sigma_{13,12}|
    \end{vmatrix}
    &
    \begin{vmatrix}
      \Sigma_{2,4} & \Sigma_{3,4}\\
      |\Sigma_{12,12}| &      |\Sigma_{13,12}|
    \end{vmatrix}
  \end{vmatrix}.
\]
Moreover, by Lemma~\ref{lem:swapping},
\begin{align*}
  |\Sigma_{12, 12}|
  &= \mathcal P_{1,1}\mathcal P_{2,2,(\{2,3,4\}, \{2,3,4\})},
  & |\Sigma_{12, 13}| &= \mathcal P_{1,1}\mathcal P_{2,3,(\{2,3,4\},
                        \{2,3,4\})},\\
  |\Sigma_{12, 14}|&=\mathcal P_{1,1}\mathcal P_{2,4,(\{2,3,4\},
                     \{2,3,4\})}.
\end{align*}
Consequently,
\begin{align*}
  &\begin{vmatrix}
      \Sigma_{2,3} & \Sigma_{3,3}\\
      |\Sigma_{12,12}| &      |\Sigma_{13,12}|
    \end{vmatrix} = \mathcal P_{1,1}\mathcal P_{2,2,(\{2,3,4\},
                         \{2,3,4\})}\mathcal P_{3,3, (\{3,4\}, \{3,4\})},\\
  &\begin{vmatrix}
      \Sigma_{2,4} & \Sigma_{3,4}\\
      |\Sigma_{12,12}| &      |\Sigma_{13,12}|
    \end{vmatrix} = \mathcal P_{1,1}\mathcal P_{2,2,(\{2,3,4\},
                         \{2,3,4\})}\mathcal P_{3,4,(\{3,4\}, \{3,4\})}.
\end{align*}
    Thus, we can write the full determinant as
    $$ f
  %   \begin{vmatrix}
  %   |\Sigma_{12,13}| & |\Sigma_{12,14}|\\ 
  %   \begin{vmatrix}
  %     \Sigma_{2,3} & \Sigma_{3,3}\\
  %     |\Sigma_{12,12}| &      |\Sigma_{13,12}|
  %   \end{vmatrix}
  %   &
  %   \begin{vmatrix}
  %     \Sigma_{2,4} & \Sigma_{3,4}\\
  %     |\Sigma_{12,12}| &    |\Sigma_{13,12}|
  %   \end{vmatrix}
  % \end{vmatrix}
    = \mathcal P_{1,1}^2\mathcal P_{2,2, (\{2,3,4\}, \{2,3,4\})}\begin{vmatrix}
    \mathcal P_{2,3, (\{2,3,4\}, \{2,3,4\})}&\mathcal P_{2,4,(\{2,3,4\}, \{2,3,4\})}\\\mathcal P_{3,3,(\{3,4\}, \{3,4\})}&\mathcal P_{3,4,(\{3,4\},\{3,4\})}\end{vmatrix}.$$
 By Lemma~\ref{lem:E_ijF_ij},
  $$ f=\mathcal P_{1,1}^2\mathcal P_{2,2, (2,3,4, 2,3,4)}\mathcal P_{\{2,3\}, \{3,4\}, (\{2,3,4\}, \{2,3,4\})}.$$
 The last term suggests that $\{2,3\}$ and $\{3,4\}$ are $(\{2,3,4\},\{2,3,4\})$-restricted trek separated. Indeed, this is the case, and they are separated by $(\emptyset, \{3\})$.
\end{example}
\begin{remark}
  Appendix B of \cite{VanOmmenMooij_UAI_17} explicitly lists the
  (minimal) generators of all vanishing ideals of acyclic mixed graphs
  on 4 nodes. Of these graphs, only those from Figure \ref{fig:ommen}
  cannot be immediately recognized as being determinantal constraints
  on the covariance matrix. From what we have shown above we now see
  that all generators of vanishing ideals of acyclic mixed
  graphs on four nodes can be written as nested determinants with, at
  most, a single level of nesting (i.e., as determinants of
  determinants), and moreover can be explained via restricted trek separation.
\end{remark}

As we also record in Section~\ref{sec:discussion}, we believe that
restricted trek separation can always be formed as a factor of a
vanishing recursively nested determinant.  Moreover, we deem it
possible that such determinants define all acyclic linear structural
equation models.  We defer further exploration of these questions to a
future study.

\subsection{Graphs with cycles}\label{sec:6.2}

Although we believe our results from the previous two sections can be extended to graphs containing cycles, the situation there is a bit more complicated. Even extending Theorem~\ref{TrekSystems} to the cyclic case is not a simple task. It was accomplished (along with other results) in a separate article~\cite{DST}.

The following model, although it has cycles, is defined by the vanishing of a nested determinant, which can be explained by restricted trek separation. However, we wish to point out that a more sophisticated example might need the definition of further notions, like those that appear in~\cite{DST}.

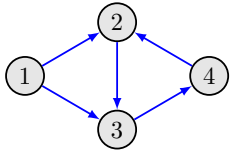
\begin{figure}[t]
  \centering
  \tikzset{
    every node/.style={circle, draw,inner sep=1mm, minimum size=0.55cm, draw, thick, black, fill=gray!20, text=black},
    every path/.style={thick}
  }
  \scalebox{0.85}{
    \begin{tikzpicture}[align=center,node distance=2.2cm]
      \node [] (2) {2};
      \node [] (1) [below left=0.4cm and 1cm of 2] {1};
      \node [] (4) [below right=0.4cm and 1cm of 2] {4};
      \node [] (3) [below right=0.4cm and 1cm of 1] {3};

      \draw[blue] [-latex] (1) edge (2);
      \draw[blue] [-latex] (2) edge (3);
      \draw[blue] [-latex] (3) edge (4);
      \draw[blue] [-latex] (1) edge (3);
      \draw[blue] [-latex] (4) edge (2);
    \end{tikzpicture}
  }
  \caption{A cyclic graph whose model is defined by a nested
    determinant.}\label{fig:cyclic4}
\end{figure}

\begin{example}
  %Example from \cite{drton:lrt}.  
  Consider the graph $G$ in Figure~\ref{fig:cyclic4}, which was treated
  in \cite{drton:lrt} where a degree 6 polynomial $f$ generating $\mathcal
  I(G)$ was displayed.   This polynomial can be written as the
  following doubly nested determinant: 
  \begin{align*}\label{eq:cyclic_graph}
    f=\begin{vmatrix}
        \left|\Sigma_{34,12}\right| & \left|\Sigma_{34,13}\right| \\
        \begin{vmatrix}
          \left| \Sigma_{12,12}\right| & \left|\Sigma_{12,34}\right|\\
          \left| \Sigma_{14,12}\right| & \left|\Sigma_{14,34}\right|
        \end{vmatrix} &
        \begin{vmatrix}
          \left| \Sigma_{12,13}\right| & \left|\Sigma_{12,34}\right|\\
          \left| \Sigma_{14,13}\right| & \left|\Sigma_{14,34}\right|
        \end{vmatrix}
      \end{vmatrix}.
  \end{align*}
  We will now show that the vanishing of this determinant corresponds
  to the fact that $\{4,2\}$ and $\{2,3\}$ are
  $(\{1,2,4\},\{2,3,4\})$-restricted-trek separated by $(\emptyset,
  \{2\})$. For our derivations we use results from~\cite{DST}, where
  subdeterminants of $\Sigma$ corresponding to graphs with cycles are
  given by rational expressions. The entries of the above matrix are:
  \begin{align*}
    &\left|\Sigma_{34, 12}\right| = \frac{(\lambda_{13} + \lambda_{12}\lambda_{23})\omega_{11}\omega_{44}\lambda_{42}}{(1- \lambda_{23}\lambda_{34}\lambda_{42})^2}=\frac{\mathcal P_{3,1, (\{1,2,3\}, \{1,2,3\})}\mathcal P_{4,2,(\{1,2,4\}, \{2,3,4\})}}{(1-\lambda_{23}\lambda_{34}\lambda_{42})},\\
    &\left|\Sigma_{34,13}\right| =\frac{(\lambda_{13} +
                     \lambda_{12}\lambda_{23})\omega_{11}\omega_{44}\lambda_{42}\lambda_{23}}{(1-
                     \lambda_{23}\lambda_{34}\lambda_{42})^2}
                     =\frac{\mathcal P_{3,1, (\{1,2,3\}, \{1,2,3\})}\mathcal
                     P_{4,3, (\{1,2,4\},
                     \{2,3,4\})}}{(1-\lambda_{23}\lambda_{34}\lambda_{42})},\\
    &                  \begin{vmatrix}
          \left| \Sigma_{12,12}\right| & \left|\Sigma_{12,34}\right|\\
          \left| \Sigma_{14,12}\right| & \left|\Sigma_{14,34}\right|
        \end{vmatrix} = %\frac{(\lambda_{13} + \lambda_{12}\lambda_{23})\omega_{11}\omega_{44}}{(1- \lambda_{23}\lambda_{34}\lambda_{42})^2}\frac{\omega_{11}\omega_{22}}{(1- \lambda_{23}\lambda_{34}\lambda_{42})^2}=
        \Sigma_{14,34}\omega_{11}\omega_{22}\frac1{1-\lambda_{23}\lambda_{34}\lambda_{42}} = 
        \Sigma_{14,34}\omega_{11}\mathcal P_{2,2, (\{1,2\}, \{2,3,4\})},\\
    &               \begin{vmatrix}
          \left| \Sigma_{12,13}\right| & \left|\Sigma_{12,34}\right|\\
          \left| \Sigma_{14,13}\right| & \left|\Sigma_{14,34}\right|
        \end{vmatrix} =
                        \Sigma_{14,34}\omega_{11}\omega_{22}\lambda_{23}\frac1{1-\lambda_{23}\lambda_{34}\lambda_{42}}
                        = \Sigma_{14,34}\omega_{11}\mathcal
                        P_{2,3,(\{1,2\}, \{2,3,4\})}.
                          \end{align*}
        It follows that
        \begin{align*}
          f&\;=\;\Sigma_{14,34}\omega_{11} \mathcal P_{3,1, (\{1,2,3\}, \{1,2,3\})}\begin{vmatrix} \mathcal P_{4,2,(\{1,2,4\}, \{2,3,4\})} & \mathcal P_{4,3, (\{1,2,4\}, \{2,3,4\})}\\
        \mathcal P_{2,2, (\{1,2\}, \{2,3,4\})} & \mathcal P_{2,3,(\{1,2\}, \{2,3,4\})}
      \end{vmatrix}\\
      &\;=\; \frac{\Sigma_{14,34}\omega_{11} \mathcal P_{3,1, (\{1,2,3\}, \{1,2,3\})}
        \mathcal P_{\{4,2\}, \{2,3\}, (\{1,2,4\},
          \{2,3,4\})}}{(1-\lambda_{23}\lambda_{34}\lambda_{42})},
    \end{align*}
            where the last equality follows by
            Lemma~\ref{lem:E_ijF_ij}. The last term in the numerator
            vanishes due to the above mentioned restricted trek
            separation. We remark that the extra term
            $(1-\lambda_{23}\lambda_{34}\lambda_{42})$ in the
            denominator cannot be obtained via the formula given in
            Theorem~\ref{thm:one_nested} (even with the usage of
            geometric series). 
\end{example}

\subsection{Nested determinants with no restricted trek separation}\label{sec:6.3}
We have not found any examples of acyclic mixed graphs $G$ for which
defining equations of the model $\mathcal{M}(G)$ cannot be explained
using restricted trek separation. However, there are other, closely
related models, for which restricted trek separation does not seem to
provide the same combinatorial explanation.

\begin{example}[The pentad] 

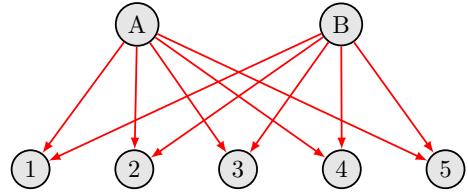
\begin{figure}[t]
  \centering
  \tikzset{
    every node/.style={circle, draw,inner sep=1mm, minimum size=0.55cm, draw, thick, black, fill=gray!20, text=black},
    every path/.style={thick}
  }
  \scalebox{0.85}{
    \begin{tikzpicture}[align=center,node distance=2.2cm]
      \node [] (1) {A};
      \node [] (2) [right=2.5cm of 1] {B};
      \node [] (3) [below left=1.8cm and 1.2cm of 1] {1};
      \node [] (4) [right= 1cm of 3] {2};
      \node [] (5) [right= 1cm of 4]{3};
      \node [] (6) [right= 1cm of 5]{4};
      \node [] (7) [right= 1cm of 6]{5};

      \draw[red] [-latex] (1) edge (3);
      \draw[red] [-latex] (1) edge (4);
      \draw[red] [-latex] (1) edge (5);
      \draw[red] [-latex] (1) edge (6);
      \draw[red] [-latex] (1) edge (7);
      \draw[red] [-latex] (2) edge (3);
      \draw[red] [-latex] (2) edge (4);
      \draw[red] [-latex] (2) edge (5);
      \draw[red] [-latex] (2) edge (6);
      \draw[red] [-latex] (2) edge (7);

%      \draw[blue] [-latex] (2) edge (3);
%      \draw[blue] [-latex] (2) edge (4);
%      \draw[red] [latex-latex] (1) edge (3);
%      \draw[red] [latex-latex] (1) edge (4);
    \end{tikzpicture}
  }
  \caption{The factor analysis model on five nodes with two factors. The vanishing ideal of this model is generated by one degree five polynomial.}\label{fig:pentad}
\end{figure}

Consider a factor analysis model with five normally distributed
observed variables and two latent factors as in Figure~\ref{fig:pentad}. Its defining equation is a degree 5 polynomial in the covariance matrix entries:
\begin{align*}
f_{\text{pentad}} &=\sigma_{12}\sigma_{13}\sigma_{24}\sigma_{35}\sigma_{45} -\sigma_{12}\sigma_{13}\sigma_{25}\sigma_{34}\sigma_{45}- \sigma_{12}\sigma_{14}\sigma_{23}\sigma_{35}\sigma_{45} + \sigma_{12}\sigma_{14}\sigma_{25}\sigma_{34}\sigma_{35}\\
&+\sigma_{12}\sigma_{15}\sigma_{23}\sigma_{34}\sigma_{45}- \sigma_{12}\sigma_{15}\sigma_{24}\sigma_{34}\sigma_{35} + \sigma_{13}\sigma_{14}\sigma_{23}\sigma_{25}\sigma_{45} - \sigma_{13}\sigma_{14}\sigma_{24}\sigma_{25}\sigma_{35}\\
&-\sigma_{13}\sigma_{15}\sigma_{23}\sigma_{24}\sigma_{45} + \sigma_{13}\sigma_{15}\sigma_{24}\sigma_{25}\sigma_{34} - \sigma_{14}\sigma_{15}\sigma_{23}\sigma_{25}\sigma_{34} + \sigma_{14}\sigma_{15}\sigma_{23}\sigma_{24}\sigma_{35}.
\end{align*}
This polynomial can be expressed in nested determinantal form as 
$$f_{\text{pentad}} = \begin{vmatrix} |\Sigma_{23, 45}| & |\Sigma_{25, 34}|\\
|\Sigma_{123, 145}| & |\Sigma_{125, 134}|
\end{vmatrix}.$$
Combinatorially, we can see that all trek systems stemming from the
second row of the matrix are in one-to-one correspondence with the
trek systems from the first row of the matrix, and are obtained by
just adding the trek $1 - 1$. However, we have not been able to
interpret this nested determinant via restricted trek separation. Note
that the mixed graph obtained by latent projection would be a complete
graph with a bidirected edge between each $i,j\in\{1,\ldots,5\}$.
\end{example}

%%% Local Variables:
%%% TeX-master: "nested_dets"
%%% End:

%\input{moreexamples}

\section{Discussion}
\label{sec:discussion}

We conclude by giving a brief review of the results presented in this
paper and then discussing problems for future work.

\subsubsection*{Contributions}

This paper demonstrates the importance of nested determinants as
constraints on covariance matrices in graphical causal/structural
equation models associated to mixed graphs.  Nested determinants are
determinants of matrices whose entries are determinants themselves.
Theorem~\ref{thm:cut-ancestral} shows that a special class of
parentally nested determinants is sufficient for a semialgebraic
description of a class of models that is slightly more general than
the class of ancestral graph models.  Theorem~\ref{thm:one_nested}
provides a framework for explaining the vanishing of more general
nested determinants via trek separation under
restrictions on the vertices that treks may visit on their left and
their right sides.

The examples from Section~\ref{sec:beyond-theorem} depict graphs for
which the conditions of Theorem~\ref{thm:one_nested} do not apply.
While it is often possible to present the defining equations of such models
in terms of (singly) nested determinants, we suggest to instead view the equations as recursively
nested determinants.  In other words, we consider determinants of
smaller matrices whose entries are (recursively) nested determinants.
As we exemplified, such recursively nested determinants may admit an
explanation by restricted trek separation. We further exhibit an
example of a graph with a cycle in which the model is also described
by a recursively nested determinant that admits a restricted trek
separation interpretation.

\subsubsection*{Definition of nested and recursively nested determinants}

Theorem~\ref{thm:one_nested} is concerned with a particular type of
nested determinants where rows and columns of the considered matrix
correspond to vertices of the graph/the given random variables.  This
setup contains as a special case the parentally nested determinants
from Section~\ref{sec:ancestral}.  We anticipate that the nested
determinants considered in Theorem~\ref{thm:one_nested} are
sufficiently general to describe mixed graph models as long as we
allow for a suitable notion of recursive nesting as encountered in the
Examples in Section~\ref{sec:beyond-theorem}.

In a general definition of recursively nested determinants, the
subdeterminants of the original covariance matrix would be recursively
nested determinants with depth of recursion zero.  At depth $k$, we
would take determinants of matrices whose entries are recursively
nested determinants of depth at most $k-1$.  However, it would be
desirable to constrain this construction such that for any recursively
nested determinant the rows and columns of the considered matrix can
be put in correspondence with two sets of vertices.  These sets of
vertices may then admit a restricted trek separation.

\begin{problem}
  Develop a notion of recursively nested determinants for which row
  and column indices are in correspondence with graph vertices.  The
  depth of recursion should be such that the subdeterminants of the
  original matrix are the only recursively nested determinants of
  depth 0.  The recursively nested determinants of depth 1 should be
  of the type encountered in Theorem~\ref{thm:one_nested}.
\end{problem}

% We record this formally.
%
% \begin{definition}
%   Let $\Sigma$ be an $m\times m$ matrix of indeterminates.  Let
%   $f\in\mathbb{R}[\Sigma]$ be a polynomial in the entries of $\Sigma$.
%   Then $f$ is a nested determinant of order $k$ if there
%   exist distinct elements $a_1,\dots,a_k\in[m]$ and distinct elements
%   $b_1,\dots,b_k\in[m]$ and subsets
%   $A_1,\dots,A_k,B_1,\dots,B_k,C_1,\dots,C_k,D_1,\dots,D_k\subset[m]$
%   such that
%   \[
%     f \;=\; \left|\left(\right)_{i,j\in[k]} \right|
%   \]
% \end{definition}

\subsubsection*{Tian decomposition}

In the introduction, after Example~\ref{ex:intro:verma}, we mentioned
Tian's graph decomposition, which may yield subgraphs whose covariance
matrix can be rationally identified from the covariance matrix for the
original graph $G$.  Trek separation in the subgraph then gives a
rational constraint.  Clearing denominators yields a polynomial in
$\mathcal I(G)$.

\begin{conjecture}
  Trek separation relations in subgraphs obtained from Tian's graph
  decomposition correspond to recursively nested determinants.
\end{conjecture}

\subsubsection*{Vanishing nested determinants}

The results we have given so far are sufficient conditions for the vanishing
of nested determinants.

\begin{problem}\label{prob:graphical_conditions}
  Using restricted trek separation, obtain graphical conditions that
  are necessary for the vanishing of the nested determinants from
  Theorem~\ref{thm:one_nested}.
\end{problem}

If a characterization of the vanishing of nested determinants is
established, it can be used to decide model equivalence questions.
More generally, it would be desirable to obtain conditions, sufficient
and necessary, for the vanishing of recursively nested determinants.
We formulate a ``hopeful'' conjecture for acyclic mixed graphs.

\begin{conjecture}\label{conj:equiv}
  The equality of two models $\mathcal M(G)$ and $\mathcal M(H)$ can
  be decided by comparing restricted trek separation relations in $G$
  and $H$.
\end{conjecture}

In all examples of graphs $G$ we inspected, the vanishing ideal
$\mathcal{I}(G)$ is in fact generated by nested or recursively nested
determinants.  

\begin{conjecture}
  The vanishing ideal $\mathcal I(G)$ can always be generated by
  recursively nested determinants.
\end{conjecture}

\subsubsection*{Computing restricted trek separation}
Assuming that restricted trek separation is what characterizes equivalence classes of models, as suggested by Conjecture~\ref{conj:equiv}, we may need to either 
use the graphical conditions from Problem~\ref{prob:graphical_conditions}  or to be able to 
compute restricted trek separation in order to find equivalent graphs.
\begin{problem} Design computationally efficient algorithms for checking/finding restricted trek separations.
\end{problem}

\subsubsection*{Feedback cycles}

Our focus was on acyclic mixed graphs, for which determinants of the
covariance matrix have expansions in terms of systems of treks without
sided intersection.  However, as the example of
Figure~\ref{fig:cyclic4} shows, (recursively) nested determinants are
also relevant for cyclic graphs.

\begin{problem}
  Generalize Theorem~\ref{thm:one_nested} to the general possibly
  cyclic case.
\end{problem}

Of course, all problems mentioned above also pertain 
to graphs with cycles.

\section*{Acknowledgements}

This work started during the 2016 Mathematics Research Communities
research conference on Algebraic Statistics.  Mathias Drton and Luca
Weihs were supported by the U.S.~National Science Foundation under
Grant No.~DMS 1712535. Elina Robeva was supported by a U.S.~National
Science Foundation Postdoctoral Fellowship No.~DMS 1703821.

%\begin{itemize}
%\item Note that we only have sufficient but not necessary conditions
%  for vanishing of nested determinants.
%\item Connection to Tian decomposition?  
%\end{itemize}
%
%
%\subsection{Open questions}
%\begin{problem} Can restricted trek separation always be explained as a factor of a nested determinant?
%\end{problem}
%\begin{problem} Can restricted trek separation explain the whole ideal $\mathcal I(G)$?
%\end{problem}
%\subsection{Nested determinants in other latent variable graphical models}
%Here we mention the pentad.

%%% Local Variables:
%%% TeX-master: "nested_dets"
%%% End:

\appendix
\allowdisplaybreaks

\section{Proofs for Section~\ref{sec:restricted-treks}: restricted trek separation}
\label{sec:restricted-treks-appendix}

This section is devoted to proving Theorem~\ref{thm:main}.  The proof
proceeds through rather minor modifications of the ideas of~\cite{STD}.

\subsection{Proof of Theorem~\ref{thm:main}({\em i}) for directed
  acyclic graphs} We begin by proving Theorem~\ref{thm:main} in the
case when $G$ is a directed acyclic graph (DAG). We extend it to acyclic mixed
graphs in the next section.  We first record the following
combinatorial interpretation of the entries of
$\left(I - \Lambda_{C, C}\right)^{-1}$ for a subset of vertices $C$.
\smallskip

\begin{proposition} Let $\mathcal P(i, j, C)$ be the set of directed
  paths from $i\in C$ to $j\in C$ that only use vertices from a subset
  $C\subseteq V$ in the directed  graph $G$. For each path $P$, define
  $\lambda^P = \prod_{i\to j\in P}\lambda_{ij}$. Then
$$\left[\left(I - \Lambda_{C, C}\right)^{-1}\right]_{ij} = \sum_{P\in\mathcal P(i,j, C)}\lambda^P.$$
\end{proposition}
\begin{proof} The claim follows from Proposition 3.1
  in~\cite{STD} if we consider the induced subgraph of $G$ with vertex
  set $C$.
  % While Proposition 3.1 in~\cite{STD} is written for
  % directed acyclic graphs, the claim also holds in the  general case
  % where we obtain a power series; see Section 3.2 in~\cite{STD}.
\end{proof}
\smallskip

When $G$ is a directed graph, the error covariance matrix $\Omega$ is
diagonal. This allows us to show the following lemma.  We emphasize
that in our discussion a determinant is zero if it is identically zero
as a polynomial/function.

\begin{lemma}\label{lem:3.2} In a directed graph consider sets of
  vertices $A\subseteq P, B\subseteq Q$ with $|A| =|B|$. Then
  $\det \Sigma^{(P,Q)}_{A, B}=0$ if and only if for
  every set $S\subseteq P\cap Q$ with $|S| = |A|=|B|$ either
  $\det\left(((I-\Lambda)_{P,P})^{-1}\right)_{S, A} = 0$ or
  $\det\left(((I-\Lambda)_{Q,Q})^{-1}\right)_{S,B} = 0$.
\end{lemma}

\begin{proof}
Since $\Sigma^{(P,Q)} = ((I - \Lambda)_{P, P})^{-T}\Omega_{P,
  Q}((I-\Lambda)_{Q, Q})^{-1}$, we have
$$\Sigma^{(P,Q)}_{A, B} = \left(((I - \Lambda)_{P, P})^{-T}\right)_{A, P}\Omega_{P, Q}\left(((I-\Lambda)_{Q, Q})^{-1}\right)_{Q, B}.$$
By  the Cauchy-Binet theorem,
$$\det \Sigma^{(P,Q)}_{A, B} = \sum_{S\subseteq P, R\subseteq Q}\det(((I-\Lambda)_{P, P})^{-T})_{A, S}\det(\Omega_{S, R})\det(((I-\Lambda)_{Q, Q})^{-1})_{R,B},$$
where the sum runs over $S$ and $R$ of cardinality $|A|=|B|$.  As $\Omega$ is diagonal, we obtain that
\begin{align*}
  \det \Sigma^{(P,Q)}_{A, B}
  &= \sum_{S\subseteq P\cap Q}\det(((I-\Lambda)_{P, P})^{-T})_{A, S}\det(\Omega_{S, S})\det(((I-\Lambda)_{Q, Q})^{-1})_{S,B}\\
  &=\sum_{S\subseteq P\cap Q}\det(((I-\Lambda)_{P, P})^{-1})_{S,A}
    \det(((I-\Lambda)_{Q, Q})^{-1})_{S,B}\prod_{s\in S}\omega_{s,s}.
\end{align*}
Since each monomial $\prod_{s\in S}\omega_{s,s}$ appears only in one term in this expansion, the result follows.
\end{proof}

We now recall the Gessel-Viennot-Lindstr\"om lemma. 

\begin{lemma}[Gessel-Viennot-Lindstr\"om
  lemma]\label{lem:GesselViennot} Suppose $G$ is a DAG with vertex set $\{1,\dots,m\}$. Let
  $A, B\subseteq \{1,\dots,m\}$ be such that $|A|=|B|=\ell$. Then
$$\det((I-\Lambda)^{-1})_{A,B} = \sum_{S\in\mathcal N(A,B)}(-1)^{S}\lambda^{S},$$
where $\mathcal N(A,B)$ is the set of all nonintersecting systems of
$\ell$ directed paths in $G$ from $A$ to $B$, and $(-1)^{S}$ is the
sign of the induced permutation of elements from $A$ to $B$. In
particular, $\det((I-\Lambda)^{-1})_{A,B}=0$ if and only if every
system of $\ell$ directed paths from $A$ to $B$ has two paths which
share a vertex.
\end{lemma}

We are going to use this lemma by restricting the original directed acyclic graph $G$ to the induced subgraphs on the subsets $P$ and $Q$. The lemma applies to all these subgraphs because they themselves are directed acyclic graphs.

Let $A\subseteq P, B\subseteq Q$ with $|A| =|B|=\ell$.  Consider a
system $\mathcal T = \{\tau_1,\dots, \tau_{\ell}\}$ of $\ell$
$(P,Q)$-restricted treks from $A\subseteq P$ to $B\subseteq Q$,
connecting the $\ell$ distinct vertices in $A$ to the $\ell$ distinct
vertices in $B$.  Let top$(\mathcal T)$ denote the multiset
$\{\Top(\tau_1,),\ldots,\Top(\tau_\ell)\}$. Here $\Top(\tau)$ is the unique
source of the trek $\tau$, i.e., the vertex contained in both the left
side and the right side of the trek. Note that the trek system
$\mathcal T$ consists of two systems of directed paths, a path system
$S_A$ from top$(\mathcal T)$ to $A$ which only uses vertices in $P$,
and a path system $S_{B}$ from top$(\mathcal T)$ to $B$ which only
uses vertices in $Q$. We say that $\mathcal T$ has a {\em sided
  intersection} if two paths in $S_A$ share a vertex or if two paths
in $S_B$ share a vertex.

\begin{proposition}\label{prop:3.4} In a DAG consider sets of vertices 
  $A\subseteq P$ and $B\subseteq Q$ with $|A|=|B|$. Then,
$$\det(\Sigma^{(P,Q)}_{A,B}) = 0$$
if and only if every system of (simple) $(P,Q)$-restricted treks from $A$ to $B$ has a sided intersection.
\end{proposition}

\begin{proof} Suppose that $\det(\Sigma^{(P,Q)}_{A,B})=0$, and let
  $\mathcal T$ be a $(P,Q)$-restricted trek system from $A$ to $B$. If
  all elements of the multiset top$(\mathcal T)$ are distinct, then
  Lemma~\ref{lem:3.2} implies that either
  $\det(((I-\Lambda)_{P,P})^{-1})_{\Top(\mathcal T), A} = 0$ or
  $\det(((I-\Lambda)_{Q,Q})^{-1})_{\Top(\mathcal T), B} = 0$. If
  top$(\mathcal T)$ has repeated elements, then these determinants are
  also zero since there are repeated rows. Thus, in both cases,
  Lemma~\ref{lem:GesselViennot} implies that there is an intersection
  in the path system from top$(\mathcal T)$ to $A$ or in the path
  system from top$(\mathcal T)$ to $B$.  Hence, $\mathcal T$ has a sided intersection.

Conversely, suppose that every $(P,Q)$-restricted trek system from $A$ to $B$ has a sided intersection, and let $S\subseteq P\cap Q$. If $R = \Top(\mathcal T)$ for some $(P,Q)$-restricted trek system $\mathcal T$ from $A$ to $B$, then either the path system from $\Top(\mathcal T)$ to $A$ or the path system from $\Top(\mathcal T)$ to $B$ has an intersection. If $R$ is not the set of top elements for some $(P,Q)$-restricted trek system $\mathcal T$ from $A$ to $B$, then there is no $P$-restricted path system connecting $R$ to $A$ or there is no $Q$-restricted path system from $R$ to $B$. In both cases, Lemma~\ref{lem:GesselViennot} implies that either $\det(((I-\Lambda)_{P,P})^{-1})_{R,A} = 0$ or $\det(((I-\Lambda)_{Q,Q})^{-1})_{R,B} = 0$. Then, Lemma~\ref{lem:3.2} implies that $\det(\Sigma^{(P,Q)}_{A,B}) = 0$.

Note that it is sufficient to check the systems of simple treks only.
Here, simple indicates that a trek has no repeated vertices.
\end{proof}

We now define a new DAG associated to $G$, denoted $\tilde G_{P,Q}$ in
order to be able to invoke the Max-Flow-Min-Cut Theorem (see
Theorem~\ref{thm:Menger}).  Let $P'=\{i':i\in P\}$ be a set of new
vertices, each being the copy of a corresponding vertex in $P$.  The
vertex set of graph $\tilde G_{P,Q}$ is $P'\cup Q$.  The edge set of
$\tilde G_{P,Q}$ includes the edge $i\to j$ for all $i,j\in Q$ such
that $i\to j$ is an edge in $G$.  Moreover, it includes the edge
$j'\to i'$ for all $i,j\in P$ such that $i\to j$ is an edge in $G$,
and the edge $i'\to i$ for all $i\in P\cap Q$.

\begin{proposition}\label{prop:3.5} The $(P,Q)$-restricted treks in $G$ from $i\in P$ to $j\in Q$ are in bijective correspondence with directed paths from $i'$ to $j$ in $\tilde G_{P,Q}$. Simple $(P,Q)$-restricted treks in $G$ from $i$ to $j$ are in bijective correspondence with directed paths from $i'$ to $j$ in $\tilde G_{P,Q}$ that use at most one edge from any pair $a\to b$ and $b'\to c'$ where $a,b\in Q$, $b,c\in P$.
\end{proposition}

\begin{proof} Every trek from $i$ to $j$ is the union of two paths with a common top, the left path in $P$, the right path in $Q$. The part of the trek from the top to $i$ corresponds to the subpath with only vertices in $P'$, and the part of the trek from the top to $j$ corresponds to the subpath with only vertices in $Q$. The unique edge of the form $k'\to k$ corresponds to the top of the trek. Excluding $a\to b$ and $b'\to c'$ implies that a trek never visits the same vertex $b$ twice.
\end{proof}

Menger's theorem, also known as the Max-Flow-Min-Cut theorem, now allows us to turn the sided crossing result on $G$ into a blocking characterization on $\tilde G_{P,Q}$.

\begin{theorem}[Vertex version of Menger's theorem]\label{thm:Menger}
  The cardinality of the largest set of vertex disjoint directed paths
  between two nonadjacent vertices $u$ and $v$ in a DAG is equal to
  the cardinality of the smallest blocking set, where a blocking set is
  a set of vertices whose removal from the graph ensures there is no
  directed path from $u$ to $v$.
\end{theorem}

\begin{proof}[Proof of Theorem~\ref{thm:main} for DAGs] We first focus
  on the case where $\det \Sigma^{(P,Q)}_{A,B} = 0$ so that the rank
  is at most $k-1$, where $k=|A|=|B|$. According to
  Proposition~\ref{prop:3.4}, every system of $k$ $(P,Q)$-restricted
  treks from $A$ to $B$ must have a sided intersection. That is, the
  number of vertex disjoint paths from $A'$ to $B$ is at most $k-1$ in
  the graph $\tilde G_{P,Q}$. We add two new vertices to $\tilde
  G_{P,Q}$, one vertex $u$ that points to each vertex in $A'$ and one
  vertex $v$ that each vertex in $B$ points to $v$. Thus, there are at
  most $k-1$ vertex disjoint paths from $u$ to $v$. Applying Menger's
  theorem, there is a blocking set $W$ in $\tilde G_{P,Q}$ of
  cardinality $|W|\le k-1$. Set $J_A = \{i\in P : i'\in W\}$ and $J_B = \{i\in Q: i\in W\}$. Then, we have that $|J_A| + |J_B|\leq k-1$, and these two sets $(P,Q)$-restricted trek-separate $A$ from $B$.

Conversely, suppose there exist sets $J_A\subseteq P$ and
$J_B\subseteq Q$ with $|J_A| + |J_B|\leq k-1$ which $(P,Q)$-restricted
trek-separate $A$ from $B$. Then $W = \{i: i\in J_B\}\cup\{i':i\in
J_A\}$ is a blocking set between $u$ and $v$ as above.  By Menger's theorem, since $|W|\leq k-1$, there is no vertex disjoint system of $k$ paths from $A'$ to $B$ in $\tilde G_{P,Q}$. Thus, every $(P,Q)$-restricted trek system from $A$ to $B$ has a sided intersection so that $\det \Sigma^{(P,Q)}_{A,B} = 0$ by Proposition~\ref{prop:3.4}.

From the special case of determinants, we deduce the general result,
because if the smallest blocking set has size $r$, there exists a
collection of $r$ disjoint paths between any subset of $A'$ and any
subset of $B$, and this is the largest possible number of paths in
such a collection.  This means that all $(r+1)\times (r+1)$ minors of
$\Sigma^{(P,Q)}_{A,B}$ are zero, but at least one $r\times r$ minor is
not zero.  Hence, $\Sigma^{(P,Q)}_{A,B}$ has rank $r$ for generic
choices of the parameters.
\end{proof}

\subsection{Proof of Theorem~\ref{thm:main}({\em i}) for mixed graphs}

A standard argument allows us to reduce to the case where there are no
bidirected edges in the graph. This can be achieved by subdividing the
bidirected edges; that is, for each bidirected edge
$i\leftrightarrow j$ in the graph, where $i\leq j$, we replace
$i\leftrightarrow j$ with a vertex $v_{i,j}$, directed edges
$v_{i,j}\to i$ and $v_{i,j}\to j$. If $i$ or $j$ lie in $P$ or $Q$,
then we add $v_{i,j}$ to $P$ or $Q$ respectively. Call the enhanced
sets $\overline P$ and $\overline Q$. The graph $\overline G$ obtained
from $G$ by subdividing all of its bidirected edges is called the {\em
  bidirected subdivision}, or {\em canonical DAG} associated to $G$.

\begin{proposition}\label{prop:3.9} Let $A\subseteq P$, $B\subseteq Q$
  be sets of vertices of a mixed graph with $|A| = |B|$.
\begin{enumerate}
\item[(i)] The matrix $\Sigma^{(P,Q)}_{A,B}$ associated to $G$ has the
  same generic rank as the matrix $\Sigma^{(\overline P,\overline Q)}_{A,B}$
  associated to $\overline G$.
\item[(ii)] There exist $J_L\subseteq P, J_R\subseteq Q$ with $|J_L| + |J_R| = r$ that $(J_L, J_R)$ $(P,Q)$-restricted trek-separates $A$ from $B$ in $G$ if and only if there exist $\overline J_L\subseteq \overline P, \overline J_R\subseteq \overline Q$ with $|\overline J_L|+|\overline J_R| = r$ that $(\overline J_L, \overline J_R)$ $(\overline P, \overline Q)$-restricted trek-separates $A$ from $B$ in $\overline G$.
\end{enumerate}
\end{proposition}

\begin{proof}
  (i)   Let $\bar\Lambda=(\bar\lambda_{k,l})$ and
  $\bar\Omega=(\bar\omega_{k,l})$ be parameters for $\bar G$.  Define
  parameters for $G=(V,\mathcal{D},\mathcal{B})$ as follows.  For any
  directed edge $i\to j$ in $G$, set
  $\lambda_{i,j}=\bar\lambda_{i,j}$.  For any bidirected edge $i\bi j$
  in $G$, set
  \begin{equation}
    \omega_{i,j} = \overline
    \omega_{v_{i,j},v_{i,j}}\overline\lambda_{v_{i,j},i}
    \overline\lambda_{v_{i,j},j}.\label{eq:subdivide-ij} 
  \end{equation}
  Finally, for each vertex $i$ in $G$, set
  \begin{equation}
    \omega_{i,i} = \overline \omega_{i,i} +
    \sum_{j\leftrightarrow
      i\in G}\overline\omega_{v_{i,j},v_{i,j}}
    \overline\lambda_{v_{i,j}i}^2.\label{eq:subdivide-ii}
  \end{equation}
  Clearly, $\Lambda=(\lambda_{i,j})\in\mathbb{R}^{\mathcal{D}}$.
  Since all $\overline\omega_{i,i}>0$, the matrix
  $\Omega=(\omega_{i,j})$ is positive definite and, thus, in
  $\mathit{PD}(\mathcal{B})$.  Let $\Sigma^{(P,Q)}_{A,B}$ be the
  matrix defined by $(\Lambda,\Omega)$, and let
  $\Sigma^{(\overline P,\overline Q)}_{A,B}$ be the matrix defined by
  $(\overline\Lambda,\overline\Omega)$.  Applying the
  $(P,Q)$-restricted trek rule to $G$ and $\bar G$, respectively, we
  see that
  $\Sigma^{(P,Q)}_{A,B}=\Sigma^{(\overline P,\overline Q)}_{A,B}$.  We
  conclude that the set of matrices
  $\Sigma^{(\overline P,\overline Q)}_{A,B}$ associated to $\bar G$ is
  contained in the set of matrices $\Sigma^{(P,Q)}_{A,B}$ associated
  to $G$.

  In general the reverse inclusion does not hold \cite{MR2740626}.
  Nevertheless, the set of matrices $\Sigma^{(P,Q)}_{A,B}$ for $G$ has
  the same Zariski closure as the set of
  $\Sigma^{(\overline P,\overline Q)}_{A,B}$ for $\overline G$.  Let
  $\mathcal{U}\subset\mathbb{R}^{\mathcal{D}}\times
  \mathit{PD}(\mathcal{B})$ be a neighborhood of $(0,I)$, i.e., we
  consider matrices $\Lambda$ with entries of small magnitude and
  $\Omega$ near the identity matrix.  To prove equality of the Zariski
  closures, it suffices to show that every matrix
  $\Sigma^{(P,Q)}_{A,B}$ given by a choice of
  $(\Lambda,\Omega)\in\mathcal{U}$ is equal to a matrix
  $\Sigma^{(\overline P,\overline Q)}_{A,B}$ associated to a choice of
  $\overline\Lambda$ and $\overline\Omega$ for $\overline G$.  This in
  turn will follow from the trek rule if we can find
  $(\overline\Lambda,\overline\Omega)$ such
  that~\eqref{eq:subdivide-ij} and~\eqref{eq:subdivide-ii} hold.
  However, this is possible because near the identity matrix, each
  off-diagonal entry $\omega_{i,j}$ is small.  Specifically, we choose
  $\overline \omega_{v_{i,j},v_{i,j}} = 1$, and set
  $\overline \lambda_{v_{i,j},i} = \sqrt{|\omega_{i,j}|}$ and
  $\overline\lambda_{v_{i,j},j} =
  \text{sign}(\omega_{i,j})\sqrt{|\omega_{i,j}|}$.  When the
  $\omega_{i,j}$ are small enough, the sum on the right-hand side
  of~\eqref{eq:subdivide-ii} is smaller than $\omega_{i,i}$, which is
  near one.  Hence, we can find a positive $\overline\omega_{i,i}$
  satisfying~\eqref{eq:subdivide-ii}, which ensures that
  $\overline\Omega$ is a diagonal matrix with positive diagonal
  entries as required.

  (ii) Any pair of sets $S_L$ and $S_R$ that are $(P,Q)$-restricted
  trek-separating in $G$ are also clearly
  $(\overline P,\overline Q)$-restricted trek-separating in
  $\overline G$.  Conversely, suppose that
  $(\overline J_L, \overline J_R)$ is a minimal
  $(\overline P,\overline Q)$-restricted trek-separating set in
  $\overline G$; that is, if any vertex is deleted from
  $(\overline J_L, \overline J_R)$, we no longer have a
  $(\overline P,\overline Q)$-restricted trek-separating set. We show
  that such a minimal $(\overline P,\overline Q)$-restricted
  trek-separating set in $\overline G$ corresponds to a
  $(P,Q)$-restricted trek-separating set in $G$. Define
$$J_L = (\overline J_L\cap P)\cup \{i\in P: v_{i,j}\in\overline J_L\},$$
$$J_R = (\overline J_R\cap Q)\cup\{j\in Q: v_{i,j}\in\overline J_R\}.$$
If $\overline J_L$ and $\overline J_R$ contain none of the vertices
$v_{i,j}$, then $J_L$ and $J_R$ clearly $(P,Q)$-restricted
trek-separate $A$ and $B$ in $G$. Otherwise, the way that
$\{i\in P:v_{i,j}\in\overline J_L\}$ and
$\{j\in Q: v_{i,j}\in \overline J_R\}$ are chosen is important. Given
a vertex $v_{i,j}\in\overline J_L\cup\overline J_R$, let
$\mathcal T(v_{i,j})$ denote the set of $(P,Q)$-restricted treks
$\tau = (\tau_L, \tau_R)$ from $A$ to $B$ such that
$\tau_L\cap\overline J_L = \{v_{i,j}\}$ or
$\tau_R\cap\overline J_R = \{v_{i,j}\}$. Since
$(\overline J_L, \overline J_R)$ is minimal, then
$\mathcal T(v_{i,j})$ must be nonempty. This implies that in every
$(P,Q)$-restricted trek $\tau = (\tau_L, \tau_R)\in\mathcal T(v_{i,j})$, up to
relabeling, $i$ occurs in $\tau_L$ (whose sink lies in $A$) and $j$
occurs in $\tau_R$ (whose sink lies in $B$). For if there were also a
trek $\tau=(\tau_L, \tau_R)$ in $\mathcal T(v_{i,j})$ which has $j$ in $\tau_L$ or
$i$ in $\tau_R$, we could patch two halves of these treks together to
find a $(P,Q)$-restricted trek from $A$ to $B$ that does not have a
sided intersection with $(\overline J_L, \overline J_R)$. So, assume
$i$ lies in $\tau_L$, and $j$ lies in $\tau_R$ for all
$(P,Q)$-restricted treks in $\mathcal T(v_{i,j})$. In this case, add
$i$ to $J_L$ whenever $v_{i,j}\in\overline J_L$, and add $j$ to $J_R$
whenever $v_{i,j}\in \overline J_R$. Then,
$|J_L| + |J_R|\leq |\overline J_L| + |\overline J_R|$, and
$(J_L, J_R)$ $(P,Q)$-restricted trek-separates $A$ from $B$ in $G$.
\end{proof}

%We may now proceed to the proof of the main theorem in this section.
To finish the proof of Theorem~\ref{thm:main}(i), note that Proposition~\ref{prop:3.9} immediately reduces the statement to the case of directed acyclic graphs, which was given in the previous subsection.

\subsection{Proof of Theorem~\ref{thm:main}({\em ii})} Using first the
Cauchy-Binet Theorem and then the Gessel-Viennot-Lindstr\"om
Lemma~\ref{lem:GesselViennot}, we have that
\begin{align*}
\det(\Sigma^{(P, Q)}_{A, B}) &=\det \left(((I - \Lambda)_{P, P})^{-T}\right)_{A, P}\Omega_{P, Q}\left(((I-\Lambda)_{Q, Q})^{-1}\right)_{Q, B}\\
&= \sum_{S\subseteq P, R\subseteq Q,\atop |S| = |R| = |A|} \det(((I-\Lambda)_{P, P})^{-T})_{A, S}\det(\Omega_{S, R})\det(((I-\Lambda)_{Q, Q})^{-1})_{R, B}\\
&=  \sum_{S\subseteq P, R\subseteq Q, \atop |S| = |R| = |A|}\,\,\sum_{\tau_1\in\mathcal N(S, A),\atop \tau_2\in\mathcal N(R, B)} (-1)^{\tau_1+\tau_2}\lambda^{\tau_1 + \tau_2}\det(\Omega_{S, R})\\
& = \sum_{S\subseteq P, R\subseteq Q, \atop |S| = |R| = |A|}\,\,\sum_{\tau_1\in\mathcal N(S, A),\atop \tau_2\in\mathcal N(R, B)}\sum_{\sigma\in\Sigma_{|S|}} (-1)^{\tau_1+\tau_2+\text{sign}(\sigma)}\lambda^{\tau_1 + \tau_2}\prod_i \omega_{s_i, r_{\sigma(i)}}.
\end{align*}
The latter sum goes over all trek systems between $A$ and $B$ whose
left directed parts have no sided intersection and only use vertices
from $P$, whose right directed parts have no sided intersection and
only use vertices from $Q$, and use left and right sides are joined
via ``middle vertices'' in $S$ and $R$.  Each summand is the product
of the trek monomials of the treks in each such system times the sign
of the permutation induced by each such trek system. Moreover, note
that each trek system with no sided intersection between $A$ and $B$
appears in this sum.

%%% Local Variables:
%%% TeX-master: "nested_dets"
%%% End:

\section{Proofs for nested determinants}

%\subsection{Proof of Lemma~\ref{RestrictedTrekSystems}}\label{app:pf_lem5.2}
%\begin{proof} Using the Cauchy-Binet Theorem, we have that
%$$\det(\Sigma^{(P, Q)}_{A, B}) =\det \left(((I - \Lambda)_{P, P})^{-T}\right)_{A, P}\Omega_{P, Q}\left(((I-\Lambda)_{Q, Q})^{-1}\right)_{Q, B}$$
%$$= \sum_{S\subseteq P, R\subseteq Q, |S| = |R| = |A|} \det(((I-\Lambda)_{P, P})^{-T})_{A, S}\det(\Omega_{S, R})\det(((I-\Lambda)_{Q, Q})^{-1})_{R, B}$$
%$$ =  \sum_{S\subseteq P, R\subseteq Q, |S| = |R| = |A|}\,\,\sum_{T_1\in\mathcal N(S, A), T_2\in\mathcal N(R, B)} (-1)^{T_1+T_2}\lambda^{T_1 + T_2}\det(\Omega_(S, R))$$
%$$ = \sum_{S\subseteq P, R\subseteq Q, |S| = |R| = |A|}\,\,\sum_{T_1\in\mathcal N(S, A), T_2\in\mathcal N(R, B)}\sum_{\sigma\in\Sigma_{|S|}} (-1)^{T_1+T_2+\text{sign}(\sigma)}\lambda^{T_1 + T_2}\prod_i \omega_{s_i, r_{\sigma(i)}}.$$
%The latter sum is precisely the sum over all trek systems between $A$ and $B$ that use middle vertices $S$ and $R$ and whose left directed parts have no sided intersection and only use vertices from $P$ and whose right directed parts have no sided intresection and only use vertices from $Q$ (by Gessel-Viennot-Lindstr\"om Lemma~\ref{lem:GesselViennot}) of the product of the trek monomials of the treks in each such system times the sign of the permutation induced by each such trek system. Moreover, note that each trek system with no sided intersection between $A$ and $B$ appears in this sum.
%\end{proof}

\subsection{Proof of Lemma~\ref{lem:swapping}}\label{app:pf_lem_swapping}

\begin{proof} Suppose first that there exist $A_i$ and $A_j$ such that
  $A_i\cap A_j \neq 0$ for $i\neq j$. The case where two of the
  $B_i$'s intersect is analogous. Then,
  $|\Sigma_{A_1\uplus\cdots\uplus A_k, B_1\uplus\cdots\uplus B_k}| =
  0$ since this matrix has a repeated row. On the other hand, we can
  select $C_i = \emptyset$, which makes $\mathcal P_{A_i, B_i, (C_i,
    D_i)} = 0$, so that for any choice of the rest of the $C_j$ and
  $D_j$, we have that $\prod_{j=1}^k\mathcal P_{A_j, B_j, (C_j,
    D_j)}=0$. Thus, both sides are equal to 0, which establishes the
  statement. 

  Now, assume that $A_i\cap A_j = B_i\cap B_j = \emptyset$ for all
  $i\neq j$. We know by Theorem~\ref{TrekSystems} that
  \[
  |\Sigma_{A_1\cup\cdots\cup A_k, B_1\cup\cdots\cup B_k}| = \mathcal
  P_{A_1\cup\cdots\cup A_k, B_1\cup\cdots\cup B_k}.
  \]
For every $i=1,\ldots,k$ let $C_i^c$, the complement of $C_i$, be the union over all treks in trek systems with
no sided intersection between $A_1\cup\cdots\cup A_k$ and
$B_1\cup\cdots\cup B_k$ of the vertices that take part in the left
side of the treks that start at
$A_1\cup\cdots\cup A_{i-1}\cup A_{i+1}\cup\cdots\cup A_k$. Let $D_i^c$
be the union over all treks in trek systems with no sided intersection between
$A_1\cup\cdots\cup A_k$ and $B_1\cup\cdots\cup B_k$ of the vertices
that take part in the right side of the treks that start at
$A_1\cup\cdots\cup A_{i-1}\cup A_{i+1}\cup\cdots\cup A_k$ (and end at
$B_1\cup\cdots\cup B_{i-1}\cup B_{i+1}\cup\cdots\cup B_k$).

By assumption, if we are given two trek systems $\mathcal T_1$ and
$\mathcal T_2$ with
no sided intersection between $A_1\cup\cdots\cup A_k$ and
$B_1\cup\cdots\cup B_k$, then we can swap the treks from $A_i$ to
$B_i$ from the first system $\mathcal T_1$ with those from the second system
$\mathcal T_2$ and obtain two other trek systems between
$A_1\cup\cdots\cup A_k$ and $B_1\cup\cdots\cup B_k$ with no sided
intersection.  Hence, each summand in
$\mathcal P_{A_1\cup\cdots\cup A_k,B_1\cup\cdots\cup B_k}$ can be
factored uniquely as a product of one element from each of
$\mathcal P_{A_i, B_i, (C_i, D_i)}$. Conversely, the product of one
element from each of $\mathcal P_{A_i, B_i, (C_i, D_i)}$ gives an
element from
$\mathcal P_{A_1\cup\cdots\cup A_k,B_1\cup\cdots\cup B_k}$. Thus,
$$|\Sigma_{A_1\cup\cdots\cup A_k,B_1\cup\cdots\cup B_k}| = \prod_{i=1}^k\mathcal P_{A_i, B_i, (C_i, D_i)},$$
as required.
\end{proof}

\subsection{Proof of Lemma~\ref{lem:E_ijF_ij}}\label{app:pf_lem_E_ijF_ij}

\begin{proof} Recall that $\mathcal P_{\{a_1,\ldots, a_n\},
    \{b_1,\dots, b_n\}, (E, F)}$ is the sum of the trek monomials of
  all trek systems with no sided intersection between $\{a_1,\ldots,
  a_n\}$ and $\{b_1,\ldots, b_n\}$ that only use $E$ on the left and
  $F$ on the right.  For each such trek system, the trek starting at
  $a_i$ only uses $E_{ij}$ on the left, and the trek ending at $b_j$
  only uses $F_{ij}$ on the right. On the other hand, the determinant
  of $(\mathcal P_{a_i, b_j, (E_{ij}, F_{ij})})_{i,j}$ is the sum of the trek monomials of all trek systems with no sided intersection between $\{a_1,\ldots, a_n\}$ and $\{b_1,\ldots, b_n\}$ for which the trek starting at $a_i$ only uses $E_{ij}$ on the left, and the trek ending at $b_j$ only uses $F_{ij}$ on the right. Therefore, the two quantities are equal.
\end{proof}

\subsection{Proof of Proposition~\ref{prop:restricted_trek_ancestral}}\label{app:pf_prop_restricted_trek_ancestral}

\begin{proof}
  We will show that the determinant of the matrix with entries
  \[
    (|\Sigma_{\pa(u)\cup\{u\}, \pa(u)\cup\{x\}})_{u\in\pa(i)\cup\{j\},
      x\in\pa(i)\cup\{i\}}
  \]
  is divisible by
  $\mathcal P_{\pa(i)\cup\{j\}, \pa(i)\cup\{i\}, (\pa(i)\cup\{j\},
    V)}$. Combinatorially, this means that there is a
  $(\pa(i)\cup\{j\}, V)$-restricted trek separation between the sets
  $\pa(i)\cup\{j\}$ and $\pa(i)\cup\{j\}$. Indeed, they are
  $(\pa(i)\cup\{j\}, V)$-restricted trek separated by
  $(\emptyset, \pa(i))$.

We begin by showing that the sets $(\pa(u), \pa(u)), (u, x)$ for $u\in\pa(i)\cup\{j\}$ and $x\in\pa(i)\cup\{i\}$ satisfy the swapping property.  Firstly, consider a system of treks with no sided intersection between $\pa(u)\cup\{u\}$ and $\pa(u)\cup\{x\}$. Suppose that in this system it is not the case that $\pa(u)$ is connected to $\pa(u)$ and $u$ is connected to $x$. Then, there must exist a trek between $u$ and an element from $\pa(u)$. Since $u$ is ancestral, the left side of this trek has to end in a directed edge. That means that the left side of this trek contains an element from $\pa(u)$, which is impossible since this creates a sided intersection on the left side of this system. Therefore, we have a contradiction, and any such trek system connects $\pa(u)$ to $\pa(u)$ and $u$ to $x$.

Now, suppose that we have two systems of treks with no sided
intersection between $\pa(u)\cup\{u\}$ and $\pa(u)\cup\{i\}$. Call
them $\mathcal T_1$ and $\mathcal T_2$.  In both systems, the treks
connecting $u$ and $x$ need to start with a bidirected edge at $u$ or
with a directed edge away from $u$ in order to avoid intersections on
the left. We need to show that we can exchange the part connecting
$\pa(u)$ to $ \pa(u)$ in $\mathcal T_1$ with the corresponding part of
$\mathcal T_2$, thereby obtaining two new trek systems with no sided
intersection. Suppose for contradiction that once we make such an
exchange, we get a sided intersection.  Then, one of the treks from
$u$ to $x$ gets a sided intersection with a trek from $\pa(u)$ to
$\pa(u)$.  Since the former trek has the form
$u (\leftrightarrow)\rightarrow\cdots\rightarrow x$, the created
intersection has to be on its right side.  Switch the tails of the two
intersecting treks.  We get a trek of the form
$u \leftrightarrow\rightarrow\cdots\rightarrow z\in\pa(u)$. But this
is a contradiction to $u$ being ancestral.
% Hence, such treks are not
% allowed.

We have shown that the sets $(\pa(u), \pa(u)), (u, x)$ for $u\in\pa(i)\cup\{j\}$ and $x\in\pa(i)\cup\{i\}$ satisfy the swapping property. We now show that 
\begin{align}\label{eq:ancestral_factorization}
|\Sigma_{\pa(u)\cup\{u\}, \pa(u)\cup\{x\}}| = \mathcal P_{\pa(u), \pa(u)}\mathcal P_{u, x, (u, V)}.
\end{align}

Note that
$|\Sigma_{\pa(u)\cup\{u\}, \pa(u)\cup\{x\}}| = \mathcal
P_{\pa(u)\cup\{u\}, \pa(u)\cup\{x\}}$. Since the sets
$(\pa(u), \pa(u)), (u, x)$ for $u\in\pa(i)\cup\{j\}$ and
$x\in\pa(i)\cup\{i\}$ satisfy the swapping property, every trek system
with no sided intersection between $\pa(u)\cup\{u\}$ and
$\pa(u)\cup\{x\}$ splits into a trek system connecting $\pa(u)$ and
$\pa(u)$ and a single trek connecting $u$ and $x$.  The latter trek
only has the vertex $u$ on its left side.  In other words, it starts
either with a bidirected edge at $u$ or with a directed edge pointing
away from $u$.

On the other hand we claim that every trek system connecting $\pa(u)$
to $\pa(u)$ with no sided intersection, and every trek from $u$ to $x$
that only has $u$ on the left can be combined into a trek system
connecting $\pa(u)\cup\{u\}$ and $\pa(u)\cup\{x\}$ with no sided
intersection.  Suppose for contradiction that the combination gives a
sided intersection.  So, there is a trek from $a\in\pa(u)$ to
$b\in\pa(u)$ that has sided intersection with the considered trek from
$u$ to $x$.  The intersection cannot be on the left since otherwise we
would have a loop $u \rightarrow\cdots\rightarrow a\rightarrow u$
which is not allowed. Thus, there is intersection on the
right. Swapping the right tails then gives a trek
$u (\leftrightarrow)\rightarrow\cdots \rightarrow b$. But since $u$ is
ancestral, we know that every trek between $u$ and its parents has to
end with a directed edge at $u$.  We have arrived at a contradiction
and, thus, the claimed combination into a trek system connecting
$\pa(u)\cup\{u\}$ and $\pa(u)\cup\{x\}$ with no sided intersection is
possible.  This proves \eqref{eq:ancestral_factorization}.

Finally, it remains to show that 
$$\det(\mathcal P_{u,x, (u, V)})_{u\in\pa(i)\cup \{j\},x\in \pa(i)\cup\{i\}} = \mathcal P_{\pa(i)\cup \{j\}, \pa(i)\cup \{i\}, (\pa(i)\cup \{j\}, V)}.$$
But this equality follows directly from Lemma~\ref{lem:E_ijF_ij}.
\end{proof}

%%% Local Variables:
%%% mode: latex
%%% TeX-master: "nested_dets"
%%% End:

\bibliographystyle{amsalpha}
\bibliography{nested_dets_references}

\end{document}